%% file: main-en.tex
\documentclass[12pt, a4paper, reqno, titlepage, onecolumn, oneside, openany]{book}

\usepackage[english]{babel}
\usepackage{srcltx}\textheight=25.0cm \textwidth=16.5cm
\newlength{\defbaselineskip}
\newcommand{\setlinespacing}[1]%
           {\setlength{\baselineskip}{#1 \defbaselineskip}}

\usepackage[compact]{titlesec}

\titleformat{\chapter}[display]{\bfseries\Large}{\filleft\MakeUppercase{\chaptertitlename} \Huge\thechapter}
{4ex}{\vspace{2ex}\filright}[\vspace{2ex}]

\titlespacing*{\chapter}{0pt}{-10pt}{10pt}

\usepackage{amsfonts}
\usepackage{amsmath}
\usepackage{amscd}
\usepackage{amsthm}
\usepackage{amssymb}
\usepackage{mathrsfs}
\usepackage{pspicture}
\usepackage{graphics}
\usepackage{pstricks}
\usepackage{pst-plot}
\usepackage{upgreek}
\usepackage{dirtytalk}

\newcommand{\length}{\operatorname{length}}
\newcommand{\diam}{\operatorname{diam}}

\newtheorem{theorem}{Theorem}[chapter]
\newtheorem{lemma}{Lemma}[chapter]
\newtheorem{corollary}{Corollary}[chapter]
\newtheorem{definition}{Definition}[chapter]

\newtheorem*{question}{Question}
\newtheorem{OldTheorem}{Theorem}

\newtheorem{innercustomtheorem}{Theorem}
\newenvironment{customtheorem}[1]
  {\renewcommand\theinnercustomtheorem{#1}\innercustomtheorem}
  {\endinnercustomtheorem}
\newtheorem{innercustomdefinition}{Definition}
\newenvironment{customdefinition}[1]
  {\renewcommand\theinnercustomdefinition{#1}\innercustomdefinition}
  {\endinnercustomdefinition}
\newtheorem{innercustomcorollary}{Corollary}
\newenvironment{customcorollary}[1]
  {\renewcommand\theinnercustomcorollary{#1}\innercustomcorollary}
  {\endinnercustomcorollary}

\DeclareMathOperator{\sgn}{sgn}
\newcommand {\V }[1]{{\rm V}\left(#1\right)}
\def\BV{{\rm BV\,}}
\def\OSC{{\rm OSC\,}}

\def\PI{\Uppi}
\def\TPI{\tilde\Uppi}
\def\FI{\varphi_*}

\def\ZM{\ensuremath{\mathcal B}} % basis name
\def\ZC{\ensuremath{\mathbb C}}
\def\DR{\ensuremath{\mathcal{DR}}}
\def\DQ{\ensuremath{\mathcal{DQ}}}

\def\ZF{\ensuremath{\mathcal F}}

\def\ZI{\ensuremath{\mathbb I}}

\def\ZN{\ensuremath{\mathbb N}}
\def\zP{\ensuremath{\mathcal P}}
\def\ZQ{\ensuremath{\mathcal Q}}
\def\ZR{\ensuremath{\mathbb R}}
\def\MR{\ensuremath{\mathcal R}}
\def\zR{\ensuremath{\mathcal R}}

\def\ZT{\ensuremath{\mathbb T}}
\def\ZZ{\ensuremath{\mathbb Z}}

\def\deff{\overset{\text{def}}{=}}

\def\supp{{\rm supp\,}}

\def\width{{\rm wd}}
\def\length{{\rm len}}
\def\a{\textsc{\char13}}
\def\q{\usefont{T1}{fxb}{o}{b}\fontsize{0.4cm}{0.4cm}\selectfont?}

\newcommand {\intmean}[2]{\frac{1}{|#2|}\int_{#2}#1(u)\,du}

\def\md#1#2\emd{
\ifx0#1\begin{equation*} #2 \end{equation*}\fi  %  single line display, no number
\ifx1#1\begin{equation}#2\end{equation}\fi   % single line display, number
\ifx2#1\begin{align*}#2\end{align*}\fi   % aligned display, no number
\ifx3#1\begin{align}#2\end{align}\fi    % aligned display, number
\ifx4#1\begin{gather*}#2\end{gather*}\fi  % multline, not align, no number
\ifx5#1\begin{gather}#2\end{gather}\fi   % multinline, not align
\ifx6#1\begin{multline*}#2\end{multline*}\fi  %  display too long for one line
\ifx7#1\begin{multline}#2\end{multline}\fi  % as above, with numbers
\ifx8#1\begin{equation*}\begin{split}#2\end{split}\end{equation*}\fi
\ifx9#1\begin{equation}\begin{split}#2\end{split}\end{equation}\fi
}
\makeatletter
\def\env@cases{%
  \let\@ifnextchar\new@ifnextchar
  \left\lbrace
  \def\arraystretch{0.9}%
  \array{@{}l@{\quad}l@{}}}
\makeatother
\newcommand{\e }[1]{(\ref{#1})}
\newcommand{\lem }[1]{Lemma \ref{#1}}
\newcommand{\trm }[1]{Theorem \ref{#1}}
\newcommand{\otrm }[1]{Theorem \ref{#1}}
\newcommand{\sect }[1]{Section \ref{#1}}
\newcommand{\chap }[1]{Chapter \ref{#1}}
\begin{document}

\thispagestyle{empty}

\input{sections/title.tex}
\setcounter{page}{1}

\clearpage

%%%%%%%%%%%%%%%%%%%%%%%%%%%%%%%%%%%%%%%%
\setcounter{tocdepth}{1}
\tableofcontents

\numberwithin{equation}{chapter}

\chapter*{Introduction}
\addcontentsline{toc}{chapter}{Introduction}
\input{sections/intro.tex}

\numberwithin{equation}{section}

\chapter{Fatou type theorems}\label{chapter-fatou}

\section{Introduction}\label{fatou-intro}
\input{sections/1.1.tex}

\section{Auxiliary lemmas}\label{fatou-lemmas}
\input{sections/1.2.tex}

\section{The case of bounded measures and $L^1$}\label{fatou-BVL1}
\input{sections/1.3.tex}

% \section{The case of $L^p,\, 1<p<\infty$}\label{fatou-Lp}
% \input{sections/1.4.tex}

\section{The case of $L^\infty$}\label{fatou-Linfty}
\input{sections/1.5.tex}

\chapter{Littlewood type theorems}\label{chapter-littlewood}

\section{Introduction}\label{littlewood-intro}
\input{sections/2.1.tex}

\section{Divergence with characteristic function}\label{littlewood-IE}
\input{sections/2.2.tex}

\section{Divergence with Blaschke product}\label{littlewood-blaschke}
\input{sections/2.3.tex}

\chapter{Differentiation bases in $\ZR^n$}\label{chapter-rectangle}

\section{Introduction}\label{rectangles-intro}
\input{sections/3.1.tex}

\section{Some definitions and auxiliary lemmas}\label{rectangles-lemmas}
\input{sections/3.2.tex}

\section{Dyadic rectangles in $\ZR^2$}\label{dyadic-R2}
\input{sections/3.3.tex}

\section{Quasi-equivalent bases in $\ZR^n$}\label{quasi-equivalent}
\input{sections/3.4.tex}

\section{Applications}\label{applications}
\input{sections/3.5.tex}

\chapter*{Conclusion}
\addcontentsline{toc}{chapter}{Conclusion}
\input{sections/conclusion.tex}

% \chapter*{Bibliography}
\cleardoublepage
\addcontentsline{toc}{chapter}{\bibname}

\input{sections/references.tex}

\end{document}

%% file: sections/title.tex
\thispagestyle{empty}

    \centerline{\bf Yerevan State University}
    % \centerline{\bf INSTITUTE OF MATHEMATICS}
    % \centerline{\bf NATIONAL ACADEMY OF SCIENCES OF ARMENIA}

    \kern 160pt

    \centerline{\large Mher H. Safaryan}

    \kern 20pt

    \centerline{\Large\bf On estimates for maximal operators}
    \centerline{\Large\bf associated with tangential regions}

    \kern 20pt

    \centerline{\large \it (A.01.01---Mathematical Analysis)}

    \kern 20pt

    \centerline{\Huge \bf THESIS}

    \kern 30pt

    \leftline{for the degree of candidate of}

    \leftline{physical mathematical sciences}

    \kern 80pt
    \rightline{Scientific advisor}

    \rightline{Doctor of phys.-math. sciences}

    \rightline{G.~A.~Karagulyan}

    \kern 70pt

    \centerline{Yerevan  - 2018}

\newpage

%% file: sections/intro.tex
The following remarkable theorems of Fatou \cite{Fat} play significant role in the study
of boundary value problems of analytic and harmonic functions.

\begin{OldTheorem}[Fatou, 1906]\label{OldFatou-1}
Any bounded analytic function on the unit disc $D=\{z\in \mathbb{C}:\, |z|<1\}$ has non-tangential limit for almost all boundary points.
\end{OldTheorem}
\begin{OldTheorem}[Fatou, 1906]\label{OldFatou-2}
If a function $\mu$ of bounded variation is differentiable at $x_0\in \ZT$, then the Poisson integral
\md0
\zP_r(x,d\mu) \deff \frac{1}{2\pi }\int_\ZT \frac{1-r^2}{1-2r\cos (x-t)+r^2} d\mu(t)\
\emd
converges non-tangentially to $\mu'(x_0)$ as $r\to 1$.
\end{OldTheorem}

These two fundamental theorems, have many applications in different mathematical theories including analytic functions, Hardy spaces, harmonic analysis, differential equations and etc. There are various generalization of these theorems in different aspects. Almost everywhere convergence over some tangential regions investigated by Nagel and Stein \cite{NaSt}, Di Biase \cite{DiBi}, Di Biase-Stokolos-Svensson-Weiss \cite{BSSW}.
Sj\"{o}gren \cite{Sog1,Sog2,Sog3}, R\"{o}nning \cite{Ron1,Ron2,Ron3}, Katkovskaya-Krotov \cite{Kro, Kro4}, Krotov \cite{Kro2, Kro3}, Brundin \cite{Bru}, Mizuta-Shimomura \cite{MizShi}, Aikawa \cite{Aik3} studied fractional Poisson integrals with respect to the fractional power of the Poisson kernel and obtained some tangential convergence properties for such integrals. More precisely they considered the integrals

\md0
\zP_r^{(1/2)}(x,f) \deff \int_\ZT P^{(1/2)}_r(x-t) f(t)\,dt = \frac{1}{c(r)}\int_\ZT [P_r(x-t)]^{1/2} f(t)\,dt,
\emd
where
\md0
P_r(x) = \frac{1-r^2}{1-2r\cos x+r^2}, \quad 0<r<1, \quad x\in\ZT
\emd
is the Poisson kernel for the unit disk and
\md0
c(r)=\int_\ZT [P_r(t)]^{1/2}\,dt \asymp (1-r)^{1/2}\log\frac{1}{1-r}
\emd
is the normalizing coefficient. Here, the notation $A\asymp B$ means double inequality $c_1 A \le B\le c_2 A$ for some positive absolute constants $c_1$ and $c_2$, which might differ in each case.

\begin{OldTheorem}[see \cite{Sog1,Ron1,Ron2}]\label{OldRonningSogren}
For any $f\in L^p(\ZT),\, 1\le p\le \infty$
\md1\label{0.frac-poisson}
\lim_{r\to 1}\zP_r^{(1/2)} (x+\theta(r),f) = f(x)
\emd
almost everywhere $x\in\ZT$, whenever
\md1\label{0.theta-bound-P0.5}
|\theta(r)|\le
\begin{cases}
c(1-r)\left(\log\frac{1}{1-r}\right)^p & \quad\text{if}\quad 1\le p<\infty,\\
c_\alpha(1-r)^\alpha,\,\text{for any }\,0<\alpha<1 & \quad\text{if}\quad p=\infty,
\end{cases}
\emd
where $c_\alpha>0$ is a constant, depended only on $\alpha$.
\end{OldTheorem}
The case of $p=1$ is proved in \cite{Sog1}, $1<p\le \infty$ is considered in \cite{Ron1}, \cite{Ron2}. Moreover, in \cite{Ron1} weak type inequalities for the maximal operator of square root Poisson integrals are established.
\begin{OldTheorem}[R\"{o}nning, 1997]\label{OldRonning}
Let $1<p<\infty$. Then the maximal operator
\md0
\zP^*_{1/2}(x,f) \deff \sup_{\substack{|\theta|<c(1-r)\left(\log\frac{1}{1-r}\right)^p\\1/2<r<1}} \zP_r^{(1/2)}(x+\theta,|f|)
\emd
is of weak type $(p,p)$.
\end{OldTheorem}
In \cite{Kro} weighted strong type inequalities for the same operators are established. Related questions were considered also in higher dimensions. Saeki \cite{Sae} studied Fatou type theorems for non-radial kernels. Kor\'anyi \cite{Kor} extended Fatou\a s theorem for the Poisson-Szeg\"{o} integral. In \cite{NaSt} Nagel and Stein proved that the Poisson integral on the upper half space of $\ZR^{n+1}$ has the boundary limit at almost every point within a certain approach region, which is not contained in any non-tangential approach regions. Sueiro \cite{Sue} extended Nagel-Stein\a s result for the Poisson-Szeg\"{o} integral. Almost everywhere convergence over tangential tress (family of curves) were investigated by Di Biase \cite{DiBi}, Di Biase-Stokolos-Svensson-Weiss \cite{BSSW}. In \cite{Kro} and \cite{Aik3} higher dimensional cases of fractional Poisson integrals are studied as well.
% In \cite{Car} higher dimensional cases of \trm{OldCarlsson} are considered too.

In \chap{chapter-fatou} we thoroughly investigate the connection between approximate identities and convergence regions. In particular, how the non-tangential convergence is connected to Poisson kernel and bounds \e{0.theta-bound-P0.5} to the square root Poisson kernel.

We introduce $\lambda(r)-$convergence, which is a generalization of non-tangential convergence in the unit disc, where $\lambda(r)$ is a function
\md1\label{0.lambdar-def}
\lambda:(0,1)\to(0,\infty) \quad\text{with}\quad \lim_{r\to1}\lambda(r)=0.
\emd
Let $\ZT = \ZR/2\pi \ZZ$ be the unit circle. For a given $x\in \ZT$ we define $\lambda(r,x)$ to be the interval $[x-\lambda(r), x+\lambda(r)]$. If $\lambda(r)\ge \pi$ we assume that $\lambda(r,x)=\ZT$. Let $F_r(x)$ be a family of functions from $L^1(\ZT)$, where $r$ varies in $(0,1)$. We say $F_r(x)$  is $\lambda(r)-$convergent at a point $x\in \ZT$ to a value $A$, if
\md0
\lim_{r\to 1}\sup_{\theta\in \lambda(r,x)}|F_r(\theta)-A|=0.
\emd
Otherwise this relation will be denoted by
\md1\label{1.lambdar-def-1}
\lim_{\stackrel{r\to 1}{\theta\in\lambda(r,x)}}F_r(\theta)=A.
\emd
We say $F_r(x)$ is $\lambda(r)-$divergent at $x\in\ZT$ if \e{1.lambdar-def-1} does not hold for any $A\in\ZR$.

There are at least two ways to interpret $\lambda(r)-$convergence. First, we can associate the function $\lambda(r)$ with regions
\md0
\Omega^x_\lambda \deff \{r e^{i\theta}\in\ZC\colon r\in(0,1),\,|\theta-x|<\lambda(r)\}\subset D,\quad x\in\ZT.
\emd
Then $\lambda(r)-$convergence for $F_r(x)$ at some point $x\in\ZT$ becomes convergence over the region $\Omega^x_\lambda$ for $\tilde{F}(r e^{ix})=F_r(x)$. It is clear, that the non-tangential convergence in the unit disc is the case of $\lambda(r)=c(1-r)$. Second, we can think of it as one dimensional \say{pointwise-uniform} convergence on $\ZT$, meaning that $\lambda(r)-$convergence at a point $x\in\ZT$ depends only on values of functions on $\lambda(r,x)$ which contracts to $x$.

Denote by $\BV(\ZT)$ the functions of bounded variation on $\ZT$. Any given function of bounded variation $\mu\in\BV(\ZT)$ defines a Borel measure on $\ZT$. We consider the family of integrals
\md1\label{1.Phi}
\Phi_r(x,d\mu) \deff \int_\ZT \varphi_r(x-t)\,d\mu(t), \quad \mu\in\BV(\ZT),
% \Phi_r(x,f) &= \int_\ZT \varphi_r(x-t)f(t)\,dt, \quad f\in L^p(\ZT),\quad 1\le p\le\infty,\label{1.Phi-f}
\emd
where $0<r<1$ and kernels $\varphi_r\in L^\infty (\ZT)$ form an \emph{approximate identity}defined as follows:

\begin{definition}
We define an approximate identity as a family $\left\{\varphi_r\right\}_{0<r<1}\subset L^{\infty}(\ZT)$ of functions satisfying the following conditions:
\begin{itemize}
\item[$\Phi1.$] $\int_\ZT \varphi_r(t)\,dt\to 1$  as $r\to 1$,
\item[$\Phi2.$] $\varphi_r^*(x) \deff \sup\limits_{|x|\le |t|\le \pi}|\varphi_r(t)|\to 0$ as $r\to 1,\quad 0<|x|\le\pi,$
\item[$\Phi3.$] $C_\varphi \deff \sup\limits_{0<r<1}\|\varphi_r^*\|_1<\infty.$
\end{itemize}
\end{definition}

Approximate identities with the above definition were investigated in \cite{Ben, Katz, KarSaf1}.
Notation $\varphi_r$ should not be confused with the classical dilation approximate identities \cite{Ste}.
In case of $\mu$ is absolutely continuous and $d\mu(t)=f(t)dt$ for some $f\in L^p(\ZT),\,1\le p\le\infty$, then the integral \e{1.Phi} will be denoted as $\Phi_r(x,f)$.

% We also consider the following family of integrals
% \md1\label{1.K}
% \zK_r(x,d\mu) = \int_\ZT K_r(x,t)d\mu(t),\quad \mu\in\BV(\ZT),
% \emd
% where $0<r<1$ and kernels $K_r\in L^\infty(\ZT^2)$ satisfy
% \begin{itemize}
% \item[$\zK1.$] $\int_\ZT K_r(x,t)\,dt\to 1$ as $r\to 1$,\quad $x\in\ZT$,
% \item[$\zK2.$] $|K_r(x,t)|\le \varphi_r(x-t),\,x,t\in\ZT$ for some functions $\varphi_r\in L^\infty (\ZT)$ with $\mathit{\Phi2}$ and $\mathit{\Phi3}$.
% \end{itemize}

Carlsson \cite{Car} obtained almost everywhere convergence result for non-negative approximate identities with regular level sets, which is defined by the following condition:
\begin{equation*}
\sup\left\{|x| \colon x\in L(r,s)\right\} \le C|L(r,s)|, \text{ for all } 0<r<1,\,s>0,
\end{equation*}
where $C$ is some constant and $L(r,s) = \left\{x\in\ZT \colon \varphi_r(x)>s \right\}$.

\begin{OldTheorem}[Carlsson, 2008]\label{OldCarlsson}
Let $\{\varphi_r(x)\ge 0\}$ be a non-negative approximate identity with regular level sets and $\rho(r) = \|\varphi_r\|_q^{-p}$, where $1\le p<\infty$ and $q=p/(p-1)$ is the conjugate index of $p$. Then for any $f\in L^p(\ZT)$
\md0
\lim_{\substack{r\to1\\|\theta|<c \rho(r)}} \Phi_r(x+\theta,f) = f(x),
\emd
almost everywhere $x\in\ZT$.
\end{OldTheorem}

Here $Mf(x)$ is the Hardy-Littlewood maximal function of $f\in L^1(\ZT)$ defined as
\md0
Mf(x) \deff \sup_{t>0}\frac{1}{2t}\int_{x-t}^{x+t}|f(u)|\,du \quad x\in\ZT.
\emd
It is well known that the maximal operator $M$ is of weak type $(1,1)$ and stong type $(p,p)$ for $1<p\le\infty$.

Although \trm{OldCarlsson} gives a general connection between approximate identities and convergence regions, we will see that it can be extended to any approximate identity without regular level sets assumption. Moreover, obtained convergence regions are shown to be optimal for a wide class of kernels.
Here, the optimality of convergence regions is considered within the regions $\Omega_\lambda^x$ with $x\in\ZT$ and $\lambda(r)$ satisfying (\ref{0.lambdar-def}). More precisely, the optimality of convergence regions is understood as the optimality of the rate of $\lambda(r)$ (when $r\to1$) ensuring almost everywhere $\lambda(r)-$convergence. The central question of \chap{chapter-fatou} is the following:
\begin{question}
For a given approximate identity $\{\varphi_r\}$ what is the necessary and sufficiant condition on $\lambda(r)$ for which
\begin{itemize}
% \item $\lim\limits_{\stackrel{r\to 1}{x\in\lambda(r,x_0)}}\zK_r\left(x,d\mu\right)=\mu'(x_0)$ almost everywhere for any $\mu\in\BV(\ZT)\,\q$
\item $\lim\limits_{\stackrel{r\to 1}{y\in\lambda(r,x)}}\Phi_r\left(x,d\mu\right)=\mu'(x)$ almost everywhere for any $\mu\in\BV(\ZT)\,\q$
\item $\lim\limits_{\stackrel{r\to 1}{y\in\lambda(r,x)}}\Phi_r\left(x,f\right) = f(x)$ almost everywhere for any $f\in L^p(\ZT),\, 1\le p\le\infty\,\q$
\end{itemize}
\end{question}

An analogous question can also be formulated for $f\in C(\ZT)$. However, in this case \lem{1.lemma-fatouC} shows that \e{0.lambdar-def} already sufficient for everywhere $\lambda(r)-$convergence.

In \sect{fatou-BVL1} we prove that the condition
\md0
\PI(\lambda,\varphi) \deff \limsup_{r\to 1}\lambda(r) \|\varphi_r\|_\infty<\infty
\emd
is necessary and sufficient for almost everywhere $\lambda(r)-$convergence of the integrals $\Phi_r(x,d\mu)$, $\mu\in\BV(\ZT)$ as well as $\Phi_r(x,f),\,f\in L^1(\ZT)$. Moreover, we prove that convergence holds at any point where $\mu$ is differentiable for the integrals $\Phi_r(x,d\mu)$ and at any Lebesgue point of $f\in L^1(\ZT)$ for the integrals $\Phi_r(x,f)$.

\begin{customdefinition}{\ref{1.regular-kernel-def}}
We say that a given approximate identity $\{\varphi_r\}$ is regular if each $\varphi_r(x)$ is non-negative, decreasing on $[0,\pi]$ and increasing on $[-\pi, 0]$.
\end{customdefinition}
Clearly, in this case the property $\mathit{\Phi3}$ is unnecessary, since it immediately follows from $\mathit{\Phi1}$.

\begin{customtheorem}{\ref{1.T1}}[see \cite{KarSaf1}]
Let $\{\varphi_r\}$  be a regular approximate identity and $\lambda(r)$ satisfies the condition $\PI(\lambda, \varphi)<\infty$. If $\mu\in \BV( \ZT)$  is differentiable at $x_0$, then
\md0
\lim_{\stackrel{r\to 1}{x\in\lambda(r,x_0)}}\Phi_r\left(x,d\mu\right)=\mu'(x_0).
\emd
\end{customtheorem}
An analogous theorem holds as well in the non-regular case of kernels, but at this time the points where \e{1.Phi} converges satisfy strong differentiability condition.
\begin{customdefinition}{\ref{1.strong-diff-mu}}
We say a given function of bounded variation $\mu$ is strong differentiable at $x_0\in \ZT$, if
there exist a number $c$ such that the variation of the function $\mu(x)-cx$ has zero derivative at $x=x_0$.
\end{customdefinition}
If $\mu$ is absolutely continuous and $d\mu(t)=f(t)dt$ then this property means that $x_0$ is a Lebesgue point for $f(x)$, i.e.
\md0
\lim_{h\to 0}\frac{1}{2h}\int_{-h}^h|f(x)-f(x_0)|dx=0.
\emd
It is well-known that strong differentiability at $x_0$ implies the existence of $\mu'(x_0)$, and any function of bounded variation is strong differentiable almost everywhere.
\begin{customtheorem}{\ref{1.T2}}[see \cite{KarSaf1}]
Let $\{\varphi_r\}$  be an arbitrary approximate identity and $\lambda(r)$ satisfies the condition $\PI(\lambda, \varphi) <\infty$. If $\mu\in\BV(\ZT)$ is strong differentiable at $x_0\in \ZT$, then
\md0
\lim_{\stackrel{r\to 1}{x\in\lambda(r,x_0)}}\Phi_r\left(x,d\mu \right)=\mu'(x_0).
\emd
\end{customtheorem}

% \begin{customtheorem}{\ref{1.T2.1}}
% Let $\{K_r\}$ be a family of kernels with properties $\mathit{\zK1, \zK2, \Phi2, \Phi3}$ and $\lambda(r)$ satisfies the condition $\PI(\lambda, \varphi) <\infty$. If $\mu\in\BV(\ZT)$ is strong differentiable at $x_0\in \ZT$, then
% \md0
% \lim_{\stackrel{r\to 1}{x\in\lambda(r,x_0)}}\mathcal{K}_r\left(x,d\mu \right)=\mu'(x_0).
% \emd
% \end{customtheorem}

The following theorem implies the sharpness of the condition $\PI(\lambda, \varphi) <\infty$ in \trm{1.T1} and \trm{1.T2}.
\begin{customtheorem}{\ref{1.T3}}[see \cite{KarSaf1}]
If $\{\varphi_r\}$ is an arbitrary approximate identity and the function $\lambda(r)$ satisfies the condition $\PI(\lambda, \varphi)=\infty$,
then there exist a function $f\in L^1(\ZT)$  such that
\md0
\limsup_{\stackrel{r\to 1}{y\in\lambda(r,x)}}\Phi_r\left(y,f\right)=\infty
\emd
for all $x\in \ZT$.
\end{customtheorem}

Thus, the condition $\PI(\lambda, \varphi)<\infty$ determines the exact rate of $\lambda(r)$ function, ensuring such convergence. It is interesting, that this rate depends only on the values $\|\varphi_r\|_\infty$. Notice that, if the kernel $\varphi_r$ coincides with the Poisson kernel $P_r$ (which is a regular approximate identity), then $\|P_r\|_\infty \asymp \frac{1}{1-r}$ and the bound $\PI(\lambda,P)<\infty$ coincides with the well-known condition
\md1\label{0.non-tangential-condition}
\limsup_{r\to1}\frac{\lambda(r)}{1-r} < \infty,
\emd
guaranteeing non-tangential convergence in the unit disk. So, \trm{1.T1} implies and generalizes Fatou\a s theorem. Furthermore, if we take the fractional Poisson kernel $P^{(1/2)}_r$ (which is regular as well), then
\md0
\|P^{(1/2)}_r\|_\infty = \frac{1}{c(r)}\|P^{1/2}_r\|_\infty \asymp \left((1-r)\log\frac{1}{1-r}\right)^{-1}
\emd
and from \trm{1.T1} we deduce \e{0.frac-poisson} when $p=1$ with an additional information about the points where the convergence occurs.

Additionally, some weak type inequalities are established for the associated maximal operator $\Phi_{\lambda}^*$, which is defined as

% In \sect{fatou-Lp}, we obtain similar results for the integrals $\Phi_r(x,f),\,f\in L^p(\ZT),\,1<p<\infty$ with condition
% \md0
% \PI_p(\lambda,\varphi) = \limsup_{r\to 1}\lambda(r) \|\varphi_r\|_\infty \FI^{p-1}(r)<\infty,
% \emd
% where
% \md0
% \FI(r) = \sup_{x\in\ZT}|x \varphi^*_r(x)|.
% \emd
% \begin{customtheorem}{\ref{1.fatouLp}}
% Let $\{\varphi_r\}$  be an arbitrary approximate identity and $\lambda(r)$ satisfies the condition $\PI_p(\lambda, \varphi) < \infty$ for some $1<p<\infty$. Then for any $f\in L^p(\ZT)$
% \md0
% \lim_{\stackrel{r\to 1}{y\in\lambda(r,x)}} \Phi_r(y,t) = f(x),
% \emd
% almost everywhere $x\in\ZT$.
% \end{customtheorem}

% The proof of this theorem is based on weak type inequality of the associated maximal operator $\Phi_{\lambda}^*$ defined as
\md1
\Phi_{\lambda}^*(x,f) \deff \sup_{\substack{|x-y|<\lambda(r)\\0<r<1}}|\Phi_r(y,f)|
= \sup_{\substack{|x-y|<\lambda(r)\\0<r<1}}\left|\int_\ZT \varphi_r(y-t)f(t)\,dt \right|.
\emd
\begin{customtheorem}{\ref{1.weakpp}}
Let $\{\varphi_r\}$ be an arbitrary approximate identity and for some $1\le p < \infty$ the function $\lambda(r)$ satisfies
\md0
\TPI_p(\lambda, \varphi) \deff \sup_{0<r<1} \lambda(r)\|\varphi_r\|_{\infty}\FI(r)^{p-1} < \infty,
\emd
where
\md0
\FI(r) \deff \sup_{x\in\ZT}|x \varphi^*_r(x)|.
\emd
Then for any $f\in L^p(\ZT)$
\md0
\Phi_{\lambda}^*(x,f) \le C \left(M|f|^p(x)\right)^{1/p},\quad x\in\ZT,
\emd
where the constant $C$ does not depend on function $f$. In particular, the operator $\Phi_{\lambda}^*$ is of weak type $(p,p)$, i.e.
\md0
\left|\left\{x\in\ZT \colon \Phi_{\lambda}^*(x,f) > t\right\}\right| \le \frac{\tilde{C}}{t^p}\|f\|_p^p
\emd
holds for any $t>0$, where constant $\tilde{C}$ does not depend on function $f$ and $t$.
\end{customtheorem}

% \begin{customtheorem}{\ref{1.fatouLp-example}}
% Let $\{\varphi_r\}$  be an arbitrary approximate identity with
% \md0
% c_\varphi = \liminf_{r\to1}\frac{1}{\FI(r)}\int_{-\mu(r)}^{\mu(r)}|\varphi_r(t)|\,dt > 0, \quad \mu(r)=\frac{\FI(r)}{\|\varphi_r\|_\infty},
% \emd
% and $\lambda(r)$ satisfies the condition $\PI_p(\lambda, \varphi)=\infty$ for some $p,\,1<p<\infty$. Then there exists a function $f\in L^p(\ZT)$ such that
% \md0
% \limsup_{\stackrel{r\to 1}{y\in\lambda(r,x)}} \Phi_r(y,f) = \infty,
% \emd
% for all $x\in\ZT$.
% \end{customtheorem}

Using the standard methods, it can be shown that these weak type inequalities imply almost everywhere $\lambda(r)-$convergence with the condition
\md0
\PI_p(\lambda,\varphi) \deff \limsup_{r\to 1}\lambda(r) \|\varphi_r\|_\infty \FI^{p-1}(r)<\infty.
\emd

As we will see in \lem{1.phi-r-bound}, the function $\FI(r)$ satisfies
\md1\label{0.phi-r-bound-ineqs}
\frac{c}{\log\|\varphi_r\|_\infty} \le \FI(r) \le C_\varphi, \quad r_0<r<1,
\emd
where $c$ is a positive absolute constant. Note that both bounds in \e{0.phi-r-bound-ineqs} are accessible. For instance, if we take the Poisson kernel $P_r(t)$ then it can be checked that $P_*(r)\asymp1$. On the other hand, if we take the square root Poisson kernel $P^{(1/2)}_r(t)$, then one can show that
\md1\label{0.square-root-poisson}
P^{(1/2)}_*(r) \asymp \left(\log\frac{1}{1-r}\right)^{-1} \asymp \frac{1}{\log\|P^{(1/2)}_r\|_\infty}.
\emd

From the first inequality of \e{0.phi-r-bound-ineqs} it follows that for any $1\le p<\infty$ the condition $\PI_p(\lambda,\varphi)<\infty$ on $\lambda(r)$ cannot be weaker than
\md0
\limsup_{r\to 1}\lambda(r) \|\varphi_r\|_\infty \left(\frac{1}{\log\|\varphi_r\|_\infty}\right)^{p-1}<\infty.
\emd
% which again depends only on values $\|\varphi_r\|_\infty$.
The second inequality of \e{0.phi-r-bound-ineqs} ensures that the multiplier $\FI(r)$ in condition $\PI_p(\lambda,\varphi)<\infty$ can only weaken that condition (in other words can only enlarge the associated region of convergence in the unit disk) if we increase $p$, i.e. condition $\PI_{p_1}<\infty$ imples $\PI_{p_2}<\infty$, whenever $1\le p_1 \le p_2 < \infty$.

Taking into account \e{0.square-root-poisson}, note that these results imply \e{0.frac-poisson} when $1<p<\infty$ as well as \otrm{OldRonning}.
% Besides, for the Poisson kernel or for the square root Poisson kernel we have $c_\varphi=1>0$, and from \trm{1.fatouLp-example} we conclude the optimality of the bound \e{0.theta-bound-P0.5} for $1<p<\infty$.
% Observe that the bound $\PI_p(\lambda,\varphi)<\infty$ for $1<p<\infty$ determines the exact convergence regions only for approximate identities satisfying $c_\varphi>0$. In fact,
Moreover, combining \lem{1.Phi*ineq1} and \lem{1.rough-bound}, we get that \trm{1.weakpp} holds if we replace the condition $\TPI_p(\lambda,\varphi)<\infty$ by
\md1\label{0.sup-rough-bound}
% \limsup_{r\to1}
\sup_{0<r<1}\lambda(r)\|\varphi_r\|_q^p < \infty,
\emd
where $q=p/(p-1)$ is the conjugate index of $p$. Thus, we obtain \otrm{OldCarlsson} for general approximate identities, not necessarily non-negative.
% It can be shown that in case of $c_\varphi=0$, the bound \e{0.sup-rough-bound} gives better convergence regions than $\PI_p(\lambda,\varphi)<\infty$ does. However, in this case it is unclear what is the exact bound for $\lambda(r)$ ensuring almost everywhere $\lambda(r)-$convergence.

In \sect{fatou-Linfty} an analogous necessary and sufficient condition will be established also for almost everywhere $\lambda(r)-$convergence of $\Phi_r(x,f),\,f\in L^\infty(\ZT)$, and this condition looks like
\md0
\PI_\infty(\lambda,\varphi) \deff \limsup_{\delta \to 0}\limsup_{r\to 1} \int_{-\delta \lambda(r)}^{\delta \lambda(r)}\varphi_r(t)dt = 0,
\emd
which contains more information about $\{\varphi_r\}$ than $\PI(\lambda,\varphi)$ does.

\begin{customtheorem}{\ref{1.T5}}[see \cite{KarSaf1}]
If $\{\varphi_r\}$ is a regular approximate identity consisting of even functions and the function $\lambda(r)$ satisfies $\PI_\infty(\lambda,\varphi)=0$, then for any $f\in L^\infty( \ZT)$ the relation
\md0
\lim_{\stackrel{r\to 1}{y\in\lambda(r,x)}}\Phi_r\left(y,f\right)=f(x)
\emd
holds at any Lebesgue point $x\in\ZT$.
\end{customtheorem}
\begin{customtheorem}{\ref{1.T6}}[see \cite{KarSaf1}]
If $\{\varphi_r\}$ is a regular approximate identity consisting of even functions and the function $\lambda(r)$ satisfies $\PI_\infty(\lambda,\varphi)>0$, then there exists a set $E\subset \ZT$, such that $\Phi_r\left(x,\ZI_E\right)$ is $\lambda(r)-$divergent at any $x\in\ZT$.
\end{customtheorem}

One can easily check that in the case of Poisson kernel $P_r(t)$, for a given function $\lambda(r)$ with \e{0.lambdar-def}, the value of $\PI_\infty(\lambda,P)$ can be either $0$ or $1$. Besides, the condition $\PI_\infty(\lambda,P)=0$ is equivalent to \e{0.non-tangential-condition}, and $\PI_\infty(\lambda,P)=1$ coincides with
\md0
\limsup_{r\to1}\frac{\lambda(r)}{1-r} = \infty.
\emd
Now suppose that $\lambda(r)$ satisfies the condition \e{0.theta-bound-P0.5} with $p=\infty$. Simple calculations show that for such $\lambda(r)$ and for the square root Poisson kernel $P^{(1/2)}_r(t)$ we have $\PI_\infty(\lambda,P^{(1/2)})=0$. Hence \trm{1.T5} implies \e{0.frac-poisson} when $p=\infty$ with an additional information about the points where the convergence occurs. Taking $\lambda(r)=(1-r)^\alpha$ with a fixed $0<\alpha<1$ we will get $\PI_\infty(\lambda,P^{(1/2)})=1-\alpha>0$, and applying \trm{1.T6} we conclude the optimality of the bound \e{0.theta-bound-P0.5} in the case $p=\infty$ too.

% === Final remarks ===
In the definition of $\lambda(r)-$convergence the range of the parameter $r$ is $(0,1)$ with the limit point $1$, that is, we consider the convergence or divergence properties when $r\to1$. We do this way in order to compare our results with the boundary properties of analytic and harmonic functions in the unit disc. Certainly it is not essential in the theorems. We could take any set $Q\subset\ZR$ with limit point $r_0$ which is either a finite number or $\infty$. We may define an approximate identity on the real line to be a family of functions $\varphi_r\in L^\infty(\ZR)\cap L^1(\ZR)$, $r>0$, which satisfies the same conditions $\mathit{\Phi1-\Phi3}$ as approximate identity on $\ZT$ does. We just need to make a little change in the condition $\mathit{\Phi2}$, that is to add $\left\|\varphi_r^*\cdot \ZI_{\{|t|\ge \delta\}}\right\|_1\to 0$ as $r\to 0$ for any $\delta>0$. In this case usually convergence is considered while $r\to 0$. Analogously, all the results \trm{1.T1}$-$\trm{1.T6} can be formulated and proved for the integrals
\md1\label{0.4-1}
% \zK_r(x,d\mu) &= \int_\ZR \zK_r(x,t)\,d\mu(t), \quad \mu\in\BV(\ZR), \quad r>0,\\
\Phi_r(x,d\mu) = \int_\ZR \varphi_r(x-t)\,d\mu(t), \quad \mu\in\BV(\ZR), \quad r>0,
% \Phi_r(x,f) &= \int_\ZR \varphi_r(x-t)f(t)\,dt, \quad f\in L^p(\ZR)\cap L^1(\ZR), \quad r>0, \quad 1\le p\le\infty,
\emd
and they can be done just repeating the proofs with miserable changes.

Any function $\Phi\in L^\infty(\ZR)\cap L^1(\ZR)$ with $\|\Phi\|_1=1$ and $\Phi^*\in L^1(\ZR)$ defines an approximate identity by
\md0
\varphi_r(x)=\frac{1}{r}\Phi\left(\frac{x}{r}\right) \quad\text{as}\quad r\to 0. 
\emd
Operators corresponding to such kernels in higher dimensional case were investigated by Stain (\cite{Ste}, p. 57). Note for such kernels we have
\md0
\|\varphi_r\|_\infty=\frac{1}{r}\|\Phi\|_\infty, \quad \FI(r)=\sup_{x\in\ZR}\left|x\Phi^*(x)\right| \le \|\Phi^*\|_1
\emd
and therefore, for $1\le p<\infty$, the condition $\PI_p(\lambda,\varphi)<\infty$ takes the form $\lambda(r)\le c\cdot r$. The case $p=\infty$ can be done in the same way as we did it for the Poisson kernel. The value $\PI_\infty(\lambda,\Phi)$ can be either $0$ or $1$, the condition $\PI_\infty(\lambda,\Phi)=0$ is equivalent to $\lambda(r)\le c\cdot r$ and the condition $\PI_\infty(\lambda,\Phi)=1$ is equivalent to $\limsup_{r\to0}\lambda(r)/r=\infty$. The bound $\lambda(r)\le c\cdot r$ characterizes the non-tangential convergence in the upper half plane and  it turns out to be a necessary and sufficient condition for almost everywhere $\lambda(r)-$convergence of the integrals \e{0.4-1}.

In addition, we would like to bring one consequence of our results, that we consider interesting.
\begin{customcorollary}{1.1}
If $\sigma_n(x,f)$ are the Fejer means of Fourier series of a function $f\in L^1(\ZT)$ and $\theta_n=O(1/n)$, then $\sigma_n(x+\theta_n,f)\to f(x)$ at any Lebesgue point $x\in \ZT$.
\end{customcorollary}
% \begin{customcorollary}{1.2}
% If $S_n(x,f)$ are the partial sums of a function $f\in L^1(0,1)$ in Franklin system and  $\theta_n=O(1/n)$, then $S_n(x+\theta_n,f)\to f(x)$ at any Lebesgue point $x\in \ZT$.
% \end{customcorollary}

% ==== transition from chapter 1 to chapter 2 ====

Littlewood \cite{Lit} made an important complement to the theorem of Fatou, proving essentiality of non-tangential approach in that theorem. The following formulation of Littlewood\a s theorem fits to the further aim of the thesis.

\begin{OldTheorem}[Littlewood, 1927]\label{OldLittlewood}
If a continuous function $\lambda:[0,1]\to \ZR$ satisfies the conditions
\md1\label{0.0-1}
\lambda(1)=0,\quad  \lim_{r\to 1}\frac{\lambda(r)}{1-r} = \infty,
\emd
then there exists a bounded analytic function $f(z),\,z\in D$, such that the boundary limit
\md0
\lim_{r\to 1}f\left(re^{i\left(x+\lambda(r)\right)}\right)
\emd
does not exist almost everywhere on $\ZT$.
\end{OldTheorem}

There are various generalization of these theorems in different aspects. A simple proof of this theorem was given by Zygmund \cite{Zyg}. In \cite{LoPi} Lohwater and Piranian proved, that in Littlewood\a s theorem almost everywhere divergence can be replaced to everywhere and the example function can be a Blaschke product. That is

\begin{OldTheorem}[Lohwater and Piranian, 1957]\label{2.OT1}
If $\lambda(r)$ is a continuous function  with \e{0.0-1}, then there exists a Blaschke product $B(z)$ such that the limit
\md0
\lim_{r\to 1}B\left(re^{i\left(x+\lambda(r)\right)}\right)
\emd
does not exist for any $x\in \ZT$.
\end{OldTheorem}

In \cite{Aik1} Aikawa  obtained a similar everywhere divergence theorem for bounded harmonic functions on the unit disk, giving a positive answer to a problem raised by Barth [\cite{Bar}, p. 551].
\begin{OldTheorem}[Aikawa, 1990]\label{OldAikawa}
If $\lambda(r)$ is a continuous function with \e{0.0-1}, then there exists a bounded harmonic function $u(z)$ on the unit disc, such that the limit
\md0
\lim_{r\to 1} u\left(re^{i\left(x+\lambda(r)\right)}\right)
\emd
does not exist for any $x\in \ZT$.
\end{OldTheorem}
As it is noticed in \cite{Aik1} this theorem implies \otrm{OldLittlewood}. Indeed, if $u(z)$ is an example of harmonic function obtained from \otrm{OldAikawa} and $v(z)$ is its harmonic conjugate, then the holomorphic function $\textrm{exp}(u + iv)$ holds the same divergence property as $u(z)$ does.

It is well known that these theorems can be also formulated in the terms of Poisson integral
\md0
\zP_r(x,f)=\frac{1}{2\pi }\int_\ZT P_r(x-t) f(t)dt,
\emd
since any bounded analytic or harmonic function on the unit disc can be written in this form, where  $f$ is either in $H^\infty$ or $L^\infty$. In addition, the proofs of these theorems are based on some properties of such functions.

Related questions were considered also in higher dimensions. Littlewood type theorems for the higher dimensional Poisson integral established by Aikawa \cite{Aik1,Aik2} and for the Poisson-Szeg\"o integral by Hakim-Sibony \cite{HaSi} and Hirata \cite{Hira}.

Notice, that \trm{1.T6} does not imply \otrm{OldLittlewood} or \otrm{OldAikawa}. It provides everywhere divergence of
\md0
\Phi_r(x+\lambda_x(r),\ZI_E)\quad\text{as}\quad r\to1,
\emd
where each function $\lambda_x:(0,1)\to(0,\infty)$ satisfies the bound $|\lambda_x(r)|\le\lambda(r)$. In \otrm{OldLittlewood} and \otrm{OldAikawa} we have stronger divergence than in \trm{1.T6}, that is, each function $\lambda_x(r)$ coincides with a given function $\lambda(r)$.

In \chap{chapter-littlewood} we generalize Littlewood\a s theorem for the integrals $\Phi_r(x,f)$ with more general kernels than approximate identities. Namely, we consider the same integrals $\Phi_r(x,f)$ with a family of kernels $\{\varphi_r\}$ satisfying
\begin{itemize}
\item[$\Phi1.$] $\int_\ZT \varphi_r(t)\,dt\to 1 \quad\text{as}\quad r\to 1$,
\item[$\Phi4.$] $\varphi_r(x)\ge0, \quad x\in\ZT,\,0<r<1,$
\item[$\Phi5.$] for any numbers $\gamma>0$ and $0<\tau<1$ there exists such $\delta>0$ that
\md0
\int_e \varphi_r(t)\,dt < \gamma, \quad 0<r<\tau
\emd
for any measurable $e\subset\ZT$ with $|e|<\delta$.
\end{itemize}
Notice, that $\mathit{\Phi5}$ is an ordinary absolute continuity condition and it is much more weaker than the condition $\mathit{\Phi2}$. For example, it is satisfied whenever
\md0
\sup_{0<r<\tau}\|\varphi_r\|_\infty < \infty, \quad 0<\tau<1.
\emd
We introduce another quantity
\md0
\PI^*(\lambda,\varphi) \deff \limsup_{\delta \to 0}\liminf_{r\to 1}\int_{-\delta \lambda(r)}^{\delta \lambda (r)}\varphi_r(t)\,dt \le \PI_\infty(\lambda,\varphi)
\emd
and prove the following theorems.

\begin{customtheorem}{\ref{2.T1}}[see \cite{KarSaf2}]
Let $\{\varphi_r\}$ be a family of kernels with $\mathit{\Phi1,\,\Phi4,\,\Phi5}$. If a function $\lambda\in C[0,1]$ satisfies the conditions $\lambda(1)= 0$ and $\PI^*(\lambda,\varphi)>1/2$, then there exists a measurable set $E\subset \ZT$ such that
\md0
\limsup_{r\to 1}\Phi_r\left(x+\lambda(r),\ZI_E\right)- \liminf_{r\to 1}\Phi_r\left(x+\lambda(r),\ZI_E\right)\ge2\PI^*-1.
\emd
\end{customtheorem}
In the case of Poisson kernel under the condition \e{0.0-1} we have $\PI^* = 1 > 1/2$. Therefore \trm{2.T1} implies the following generalization of \otrm{OldLittlewood} and \otrm{OldAikawa}, giving additional information about the divergence character.
\begin{customcorollary}{2.1}
For any function $\lambda\in C[0,1]$ satisfying \e{0.0-1}, there exists a harmonic function $u(z),\,z\in D$ on the unit disc with $0\le u(z)\le 1$, such that
\md0
\limsup_{r\to 1} u\left(re^{i\left(x+\lambda(r)\right)}\right)=1,\quad \liminf_{r\to 1} u\left(re^{i\left(x+\lambda(r)\right)}\right)=0,
\emd
at any point $x\in \ZT$.
\end{customcorollary}

The higher dimensional case of this corollary was considered by Hirata \cite{Hira}. We construct also a Blaschke product with Littlewood type divergence condition as in \trm{2.T1}, which generalizes \otrm{2.OT1}. In this case a stronger condition $\PI^*(\lambda,\varphi)=1$ is required.
\begin{customtheorem}{\ref{2.T2}}[see \cite{KarSaf2}]
Let a family of kernels $\{\varphi_r\}$ satisfies $\mathit{\Phi1,\,\Phi4,\,\Phi5}$ and for $\lambda\in C[0,1]$ we have $\lambda(1)= 0$ and $\PI^*(\lambda,\varphi)=1$. Then there exists a function $B\in L^\infty(\ZT)$, which is the boundary function of a Blaschke product, such that the limit
\md0
\lim_{r\to 1}\Phi_r\left(x+\lambda(r),B\right)
\emd
does not exist for any $x\in\ZT$.
\end{customtheorem}

Note that, as \trm{1.T1}$-$\trm{1.T6}, \trm{2.T1} can also be formulated and proved for the integrals
\md1\label{0.scaling-operators}
\Phi_r(x,f) = \int_\ZR \varphi_r(x-t)f(t)\,dt, \quad f\in L^1(\ZR), \quad 0<r<1,
\emd
where the kernels $\varphi_r\in L^\infty(\ZR)\cap L^1(\ZR)$ satisfy the conditions $\mathit{\Phi1,\,\Phi4,\,\Phi5}$. Furthermore, notice that for any positive function $\Phi\in L^\infty(\ZR)\cap L^1(\ZR)$ with $\|\Phi\|_1=1$ the kernels
\md1\label{0.scaling-kernels}
\varphi_r(x) = \frac{1}{1-r}\Phi\left(\frac{x}{1-r}\right), \quad x\in\ZR, \quad 0<r<1
\emd
satisfy the conditions $\mathit{\Phi4}$ and $\mathit{\Phi5}$. One can check, that for the Poisson kernel and for \e{0.scaling-kernels} the following  conditions are equivalent
\md0
\lim_{r\to1}\frac{\lambda(r)}{1-r}=\infty \iff \PI^*(\lambda,\varphi)=1 \iff \PI^*(\lambda,\varphi)>0.
\emd
Therefore, if the kernels in \e{0.scaling-operators} coincide with \e{0.scaling-kernels} and $\lambda(r)$ satisfies \e{0.0-1}, then \trm{2.T1} formulated for the integrals \e{0.scaling-operators} implies everywhere strong-type divergence for \e{0.scaling-operators}, which covers the one-dimensional case of a theorem obtained by Aikawa in \cite{Aik3}.

% ==== transition from chapter 2 to chapter 3 ====

Now we proceed to the introduction of the third chapter.

Let $\zR^n$  be the family of half-open (or half-closed) rectangles $\prod\limits_{i=1}^n[a_i,b_i)$ in $\ZR^n$ and $\DR^n$ be the family of dyadic rectangles of the form
\md1\label{0.100}
\prod_{i=1}^n \left[\frac{j_i-1}{2^{m_i}},\frac{j_i}{2^{m_i}}\right),\quad j_i, m_i\in \ZZ, \quad i=1, 2, \ldots, n.
\emd
Let $\ZQ^n\subset \zR^n$ be the family of half-open squares in $\ZR^n$ and $\DQ^n$ be the family of dyadic squares ($m_1=m_2=\dots=m_n$). Obviously $\DR^n\subset \zR^n$ and $\DQ^n\subset \ZQ^n$. For a set $E\subset\ZR^n$ we denote
\md0
\diam(E) = \sup_{x,y\in E} \|x-y\|.
\emd
\begin{customdefinition}{\ref{3.def-diff-bases}}
A family $\ZM$ of bounded, positively measured sets from $\ZR^n$ is said to be a differentiation basis (or simply basis), if for any point $x\in \ZR^n$ there exists a sequence of sets $E_k\in\ZM$ such that $x\in E_k$,  $k=1,2,\ldots$ and $\diam(E_k)\to 0$ as $k\to\infty$.
\end{customdefinition}
Let $\ZM$ be a differentiation basis and $L_{\rm loc}(\ZR^n)$ be the space of locally integrable functions:
\md0
L_{\rm loc}(\ZR^n) \deff \{f \colon f\in L(K) \text{ for any compact } K\subset\ZR^n\}.
\emd
For any function $f\in L_{\rm loc}(\ZR^n)$ we define
\md0
\delta_\ZM(x,f) \deff \limsup\limits_{\diam(E)\to 0,\, x\in E\in\ZM }\left| \frac{1}{|E|}\int_E f(t)dt-f(x)\right|.
\emd
The integral of a function $f\in L_{\rm loc}(\ZR^n)$ is said to be differentiable at a point $x\in \ZR^n$  with respect to the basis  $\ZM$, if $\delta_\ZM(x,f)=0$. The integral of a function is said to be differentiable with respect to the basis $\ZM$, if it is differentiable at almost every point. Consider the following classes of functions
\md4
\ZF(\ZM) \deff \{f\in L_{\rm loc}(\ZR^n):\,\delta_\ZM (x,f)=0 \text { almost everywhere }\},\\
\ZF^+(\ZM) \deff \{f\in L_{\rm loc}(\ZR^n):\,f(x)\ge 0,\, \delta_\ZM (x,f)=0 \text { almost everywhere }\}.
\emd
Note that $\ZF(\ZM)\,\left(\ZF^+(\ZM)\right)$ is the family of (positive) functions having almost everywhere differentiable integrals with respect to the basis $\ZM$.

Let $\Psi:\ZR^+\to\ZR^+$ be a convex function.
Denote by $\Psi(L) (\ZR^n)$ the class of measurable functions $f$ defined on $\ZR^n$ such that
$\Psi(|f|)\in L^1(\ZR^n)$. If $\Phi $ satisfies the
$\Delta_2$-condition $\Psi(2x)\le k\Psi(x)$, then $\Psi(L)$ turns to be an Orlicz space with the norm
\md0
\|f\|_\Psi \deff \inf\left\{c>0:\, \int_{\ZR^n}\Psi\left( \frac{|f|}{c}\right)\le 1\right\}.
\emd
The following classical theorems determine the optimal Orlicz space, which functions have a.e. differentiable integrals with respect to the entire family of rectangles $\zR^n$ is the space
\md0
L(1+\log^+ L)^{n-1}(\ZR^n)\subset L^1(\ZR^n),
\emd
corresponding to the case $\Psi(t)=t(1+\log^+t)^{n-1}$ (\cite{Guz}).

\begin{OldTheorem}[Jessen--Marcinkiewicz--Zygmund, \cite{JMZ}]\label{OldJMZ}
\md0
L(1+\log^+ L)^{n-1}(\ZR^n)\subset \ZF(\zR^n).
\emd
\end{OldTheorem}
\begin{OldTheorem}[Saks, \cite{Saks}]\label{OldSaks}
If the function $\Psi$ satisfies
\md0
\Psi(t)=o(t\log^{n-1} t)\text { as } t\to\infty,
\emd
then $\Psi(L)(\ZR^n)\not\subset \ZF(\zR^n)$. Moreover, there exists a positive function $f\in \Psi(L)(\ZR^n)$ such that $\delta_{\zR^n}(x,f)=\infty$ everywhere.	
\end{OldTheorem}
Such theorems are valid also for the basis $\DR^n$. The first one trivially follows from embedding
\md0
L(1+\log^+ L)^{n-1}(\ZR^n)\subset \ZF(\zR^n)\subset \ZF(\DR^n).
\emd
The second can be deduced from the following
\begin{OldTheorem}[Zerekidze, \cite{Zer1} (see also \cite{Zer2, Zer3})]\label{zerekidze}
$\ZF^+(\DR^n)= \ZF^+(\zR^n)$.
\end{OldTheorem}

Let $\Delta=\{\nu_k:\,k=1,2,\ldots\}$ be an increasing sequence of positive integers. This sequence generates rare basis $\DR^n_\Delta$ of dyadic rectangles of the form \e{0.100} with $m_i\in \Delta,\,i=1, 2,\ldots,n$. This kind of bases first considered in the papers \cite{Sto}, \cite{Hag}, \cite{HarSto}, \cite{KarKarSaf}. Stokolos \cite{Sto} proved that the analogous of Saks theorem holds for any basis $\DR^n_\Delta$ with an arbitrary $\Delta$ sequence. That means $ L(1+\log^+ L)^{n-1}(\ZR^n)$ is again the largest Orlicz space containing in $\ZF(\DR^n_\Delta)$. Oniani and Zerekidze \cite{OniZer} characterised translation invariant as well as net type bases formed of rectangles that are equivalent to the basis of all rectangles in the class of all non-negative functions. Karagulyan \cite{Kar} proved some theorems, establishing an equivalency of some convergence conditions for multiple martingale sequences, those in particular imply some results of the papers \cite{Sto}, \cite{Hag}, \cite{HarSto}.

In spite of the largest Orlicz spaces corresponding to the bases $\DR^2_\Delta$ and $\DR^2$ coincide, they do differentiate different set of functions, depending on density of the sequence $\Delta$. In \sect{dyadic-R2} we prove that the condition
\md0
\gamma_\Delta \deff \sup_{k\in \ZN}(\nu_{k+1}-\nu_k) < \infty
\emd
is necessary and sufficient for the full equivalency of rare dyadic basis $\DR^2_\Delta$ and complete dyadic basis $\DR^2$.
\begin{customtheorem}{\ref{TX2}}[see \cite{KarKarSaf}]
If $\Delta=\{\nu_k\}$ is an increasing sequence of positive integers with $\gamma_\Delta<\infty$, then
\md0
\ZF(\DR^2_\Delta)= \ZF(\DR^2).
\emd
\end{customtheorem}

\begin{customtheorem}{\ref{TX2.1}}[see \cite{KarKarSaf}]
If $\Delta=\{\nu_k\}$ is an increasing sequence of positive integers with $\gamma_\Delta=\infty$, then there exists a function $f\in\ZF(\DR^2_\Delta)$ such that
\md0
\limsup_{\length(R)\to 0,\, x\in R\in\DR^2}\left|\frac{1}{|R|}\int_R f(t)\,dt\right| = \infty
\emd
for any $x\in\ZR^n$.
\end{customtheorem}

\begin{customdefinition}{\ref{3.def-density-bases}}
A basis $\ZM$ is said to be density basis if $\ZM$ differentiates the integral of any characteristic function $\ZI_E$ of measurable set $E$:
\md0
\delta_\ZM(x,\ZI_E)=0 \text{ at almost every } x\in\ZR^n.
\emd
We will say that the basis $\ZM$ differentiates a class of functions $\ZF$, if basis $\ZM$ differentiates the integrals of all functions of $\ZF$.
\end{customdefinition}

\begin{OldTheorem} [\cite{Guz}, III, Theorem 1.4]
If $\ZM$ is a density basis, then it differentiates $L^\infty$
\end{OldTheorem}

Note that any subbasis $\ZM'$ of a density basis $\ZM$ is also density basis, since in this case $\delta_{\ZM'}(x,f)\le\delta_{\ZM}(x,f)$ for any $x\in\ZR^n$ and $f\in L_{loc}(\ZR^n)$.

\begin{customdefinition}{\ref{3.def-quasi-equiv-bases}}
Let $\ZM_1,\ZM_2\subseteq\ZM$ be subbases. We will say that basis $\ZM_2$ is quasi-coverable by basis $\ZM_1$ (with respect to basis $\ZM$) if for any $R\in\ZM_2$ there exist $R_k\in\ZM_1,\,k=1, 2, \ldots, p$ and $R'\in\ZM$ such that
\md2
& R\subseteq \tilde{R} \subseteq R',\quad \tilde{R}=\bigcup_{k=1}^p R_k\\
& \diam(R')\le c\cdot\diam(R),\quad |R'|\le c|R_k|,\quad k=1, 2, \ldots, p,\\
& \sum_{k=1}^p|R_k|\le c|\tilde{R}|,\quad |\tilde{R}|\le c|R|,
\emd
where constant $c\ge1$ depends only on bases $\ZM_1, \ZM_2$ and $\ZM$. We will say two bases are quasi-equivalent if they are quasi-coverable with respect to each other.
\end{customdefinition}

In \sect{quasi-equivalent} we prove that quasi-equivalent subbases $\ZM_1, \ZM_2$ of density basis $\ZM$ differentiate the same class of non-negative functions. In \sect{applications} we give several corollaries from this theorem for bases formed of rectangles.

\begin{customtheorem}{\ref{quasi-equiv}}[see \cite{Saf}]
Let $\ZM_1$ and $\ZM_2$ be subbases of density basis $\ZM$ formed of open sets from $\ZR^n$. If the bases $\ZM_1$ and $\ZM_2$ are quasi-equivalent with respect to $\ZM$ then
\md0
\ZF^+(\ZM_1)= \ZF^+(\ZM_2).
\emd
\end{customtheorem}

Main results of the thesis are published in \cite{KarKarSaf, KarSaf2, KarSaf1, Saf, Saf2}.

%% file: sections/1.1.tex
In this chapter we generalize Fatou\a s theorem for the integrals with general kernels. Here we remind Fatou\a s theorems about non-tangential convergence of Poisson integrals and related tangential convergence results for the square root Poisson integrals as well as weak type inequalities.

\begin{customtheorem}{\ref{OldFatou-1}}[Fatou, 1906]
Any bounded analytic function on the unit disc $D=\{z\in \mathbb{C}:\, |z|<1\}$ has nontangential limit for almost all boundary points.
\end{customtheorem}
\begin{customtheorem}{\ref{OldFatou-2}}[Fatou, 1906]
If a function $\mu$ of bounded variation is differentiable at $x_0\in \ZT$, then the Poisson integral
\md0
\zP_r(x,d\mu)=\frac{1}{2\pi }\int_\ZT \frac{1-r^2}{1-2r\cos (x-t)+r^2} d\mu(t)\
\emd
converges non-tangentially to $\mu'(x_0)$ as $r\to 1$.
\end{customtheorem}
 
\begin{customtheorem}{\ref{OldRonningSogren}}[see \cite{Sog1,Ron1,Ron2}]
For any $f\in L^p(\ZT),\, 1\le p\le \infty$
\md1\label{1.frac-poisson}
\lim_{r\to 1}\zP_r^{(1/2)} (x+\theta(r),f) = f(x)
\emd
almost everywhere $x\in\ZT$, whenever
\md1\label{1.theta-bound-P0.5}
|\theta(r)|\le
\begin{cases}
c(1-r)\left(\log\frac{1}{1-r}\right)^p & \quad\text{if}\quad 1\le p<\infty,\\
c_\alpha(1-r)^\alpha,\,\text{for any }\,0<\alpha<1 & \quad\text{if}\quad p=\infty,
\end{cases}
\emd
where $c_\alpha>0$ is a constant, depended only on $\alpha$.
\end{customtheorem}

\begin{customtheorem}{\ref{OldRonning}}[R\"{o}nning, 1997]
Let $1<p<\infty$. Then the maximal operator corresponding to the square root Poisson kernel
\md0
\zP^*_{1/2}(x,f) = \sup_{\substack{|\theta|<c(1-r)\left(\log\frac{1}{1-r}\right)^p\\1/2<r<1}} \zP_r^{(1/2)}(x+\theta,|f|)
\emd
is of weak type $(p,p)$.
\end{customtheorem}

\begin{customtheorem}{\ref{OldCarlsson}}[Carlsson, 2008]
Let $\{\varphi_r(x)\ge 0\}$ be a non-negative approximate identity with regular level sets and $\rho(r) = \|\varphi_r\|_q^{-p}$, where $1\le p<\infty$ and $q=p/(p-1)$ is the conjugate index of $p$. Then for any $f\in L^p(\ZT)$
\md0
\lim_{\substack{r\to1\\|\theta|<c \rho(r)}} \Phi_r(x+\theta,f) = f(x),
\emd
almost everywhere $x\in\ZT$.
\end{customtheorem}

The organization of the current chpater is as follows. In \sect{fatou-lemmas} we prove auxiliarry lemmas, which will be used throughout the chapter. In \sect{fatou-BVL1} we prove that the condition $\PI(\lambda,\varphi)<\infty$ determines the exact rate of $\lambda(r)$ for functional spaces $\BV(\ZT)$ and $L^1(\ZT)$.

\begin{definition}\label{1.regular-kernel-def}
We say that a given approximation of identity $\{\varphi_r\}$ is regular if each $\varphi_r(x)$ is non-negative, decreasing on $[0,\pi]$ and increasing on $[-\pi, 0]$.
\end{definition}

\begin{theorem}[see \cite{KarSaf1}]\label{1.T1}
Let $\{\varphi_r\}$  be a regular approximate identity and $\lambda(r)$ satisfies the condition
\md0
\PI(\lambda, \varphi) = \limsup_{r\to 1}\lambda(r) \|\varphi_r\|_\infty<\infty.
\emd
If $\mu\in \BV( \ZT)$  is differentiable at $x_0$, then
\md0
\lim_{\stackrel{r\to 1}{x\in\lambda(r,x_0)}}\Phi_r\left(x,d\mu\right)=\mu'(x_0).
\emd
\end{theorem}

\begin{definition}\label{1.strong-diff-mu}
We say a given function of bounded variation $\mu$ is strong differentiable at $x_0\in \ZT$, if
there exist a number $c$ such that the variation of the function $\mu(x)-cx$ has zero derivative at $x=x_0$.
\end{definition}

\begin{theorem}[see \cite{KarSaf1}]\label{1.T2}
Let $\{\varphi_r\}$  be an arbitrary approximate identity and $\lambda(r)$ satisfies the condition $\PI(\lambda,\varphi)<\infty$. If $\mu\in\BV(\ZT)$ is strong differentiable at $x_0\in \ZT$, then
\md0
\lim_{\stackrel{r\to 1}{x\in\lambda(r,x_0)}}\Phi_r\left(x,d\mu \right)=\mu'(x_0).
\emd
\end{theorem}

% \begin{theorem}\label{1.T2.1}
% Let $\{K_r\}$ be a family of kernels with properties $\mathit{\zK1, \zK2, \Phi2, \Phi3}$ and $\lambda(r)$ satisfies the condition $\PI(\lambda,\varphi)<\infty$. If $\mu\in\BV(\ZT)$ is strong differentiable at $x_0\in \ZT$, then
% \md0
% \lim_{\stackrel{r\to 1}{x\in\lambda(r,x_0)}}\mathcal{K}_r\left(x,d\mu \right)=\mu'(x_0).
% \emd
% \end{theorem}

\begin{theorem}[see \cite{KarSaf1}]\label{1.T3}
If $\{\varphi_r\}$ is an arbitrary approximate identity and the function $\lambda(r)$ satisfies the condition $\PI(\lambda,\varphi)=\infty$, then there exist a function $f\in L^1(\ZT)$  such that
\md1\label{1.2-2}
\limsup_{\stackrel{r\to 1}{y\in\lambda(r,x)}}\Phi_r\left(y,f\right)=\infty
\emd
for all $x\in \ZT$.
\end{theorem}

Additionally, we prove that the bound $\TPI_p(\lambda,\varphi)<\infty$ provides weak type inequalities in spaces $L^p(\ZT),\,1\le p<\infty$.
% Then, using these weak type inequalities, we deduce that the condition $\PI_p(\lambda,\varphi)<\infty$ identifies the exact convergence regions for spaces $L^p(\ZT),\,1<p<\infty$.

% \begin{theorem}\label{1.fatouLp}
% Let $\{\varphi_r\}$  be an arbitrary approximate identity, $1<p<\infty$ and $\lambda(r)$ satisfies the condition
% \md1\label{1.limsupboundLp}
% \PI_p(\lambda, \varphi) = \limsup_{r\to 1} \lambda(r)\|\varphi_r\|_{\infty}\FI(r)^{p-1} < \infty.
% \emd
% Then for any $f\in L^p(\ZT)$
% \md0
% \lim_{\stackrel{r\to 1}{y\in\lambda(r,x)}} \Phi_r(y,t) = f(x),
% \emd
% almost everywhere $x\in\ZT$.
% \end{theorem}

\begin{theorem}[see \cite{Saf2}]\label{1.weakpp}
Let $\{\varphi_r\}$ be an arbitrary approximate identity and for some $1\le p < \infty$ the function $\lambda(r)$ satisfies
\md1\label{1.supboundLp}
\TPI_p(\lambda, \varphi) = \sup_{0<r<1} \lambda(r)\|\varphi_r\|_{\infty}\FI(r)^{p-1} < \infty.
\emd
Then for any $f\in L^p(\ZT)$
\md1\label{1.Phi*-bound}
\Phi_{\lambda}^*(x,f) \le C \left(M|f|^p(x)\right)^{1/p},\quad x\in\ZT,
\emd
where the constant $C$ does not depend on function $f$. In particular, the operator $\Phi_{\lambda}^*$ is of weak type $(p,p)$, i.e.
\md0
\left|\left\{x\in\ZT \colon \Phi_{\lambda}^*(x,f) > t\right\}\right| \le \frac{\tilde{C}}{t^p}\|f\|_p^p
\emd
holds for any $t>0$, where constant $\tilde{C}$ does not depend on function $f$ and $t$.
\end{theorem}

% \begin{theorem}\label{1.fatouLp-example}
% Let $\{\varphi_r\}$  be an arbitrary approximate identity with
% \md1\label{1.extraCondition-phi}
% c_\varphi = \liminf_{r\to1}\frac{1}{\FI(r)}\int_{-\mu(r)}^{\mu(r)}|\varphi_r(t)|\,dt > 0,\quad \mu(r)=\frac{\FI(r)}{\|\varphi_r\|_\infty},
% \emd
% and $\lambda(r)$ satisfies the condition $\PI_p(\lambda, \varphi)=\infty$ for some $p,\,1<p<\infty$. Then there exists a function $f\in L^p(\ZT)$ such that
% \md1\label{1.Lp-divergence}
% \limsup_{\stackrel{r\to 1}{y\in\lambda(r,x)}} \Phi_r(y,f) = \infty,
% \emd
% for all $x\in\ZT$.
% \end{theorem}

In \sect{fatou-Linfty} we prove that the condition $\PI_\infty(\lambda,\varphi)=0$ is necessary and sufficient for almost everywhere $\lambda(r)-$convergence of the integrals $\Phi_r(x, f),\,f\in L^\infty(\ZT)$.
\begin{theorem}[see \cite{KarSaf1}]\label{1.T5}
If $\{\varphi_r\}$ is a regular approximate identity consisting of even functions and
\md0
\PI_\infty(\lambda,\varphi) = \limsup_{\delta \to 0}\limsup_{r\to 1} \int_{-\delta \lambda(r)}^{\delta \lambda(r)}\varphi_r(t)dt = 0,
\emd
then for any $f\in L^\infty( \ZT)$ the relation
\md0
\lim_{\stackrel{r\to 1}{y\in\lambda(r,x)}}\Phi_r\left(y,f\right)=f(x)
\emd
holds at any Lebesgue point $x\in\ZT$.
\end{theorem}
\begin{theorem}[see \cite{KarSaf1}]\label{1.T6}
If $\{\varphi_r\}$ is a regular approximate identity consisting of even functions and $\PI_\infty(\lambda,\varphi)>0$, then there exists a set $E\subset \ZT$, such that
$\Phi_r\left(x,\ZI_E\right)$ is $\lambda(r)-$divergent at any $x\in\ZT$.
\end{theorem}

%% file: sections/1.2.tex
The following lemma plays significant role in the proofs of \trm{1.T1} and \trm{1.T2}.
\begin{lemma}\label{1.L1-}
Let a positive function $\varphi\in L^\infty(\ZT)$ is decreasing on $[0,\pi]$ and increasing on $[-\pi,0]$. Then for any numbers $\varepsilon\in (0,1)$ and $\theta\in (-\pi,\pi)$ there exist a finite family of intervals $I_j\subset \ZT$, $j=1,2,\ldots ,n$, containing $0$ in their closures $\bar I_j$,  and numbers $\varepsilon_j=\pm \varepsilon$ such that
\md2
&|I_j|\le 2\sup\{|t|:\, \varphi(t)\ge\varepsilon \},\quad j=1,2,\ldots,n,\\
&\sum_{j=1}^n|I_j|<10\varepsilon^{-1}\max\{1,|\theta|\cdot \|\varphi \|_\infty,\|\varphi \|_1\},\\
&\left|\varphi(x-\theta)-\sum_{j=1}^n\varepsilon_j \ZI_{I_j}(x)\right|\le\varepsilon.
\emd
\end{lemma}
\begin{proof}
Denote
\md2
&y_k=\sup\{t>0:\, \varphi(t)\ge\varepsilon k\},\\
&x_k=\sup\{t>0:\, \varphi(-t)\ge\varepsilon k\},\quad
 k=1,2,\ldots, l=\left[\frac{\|\varphi\|_\infty}{\varepsilon}\right].
\emd
Then we obviously have
\md3
&y_0=\pi,\quad 0\le y_l\le y_{l-1}\le  \ldots \le y_1\le \sup\{|t|:\, \varphi(t)\ge\varepsilon \},\label{1.2-3}\\
&x_0=\pi,\quad 0\le x_l\le x_{l-1}\le  \ldots \le x_1\le \sup\{|t|:\, \varphi(t)\ge\varepsilon \},\label{1.2-4}\\
&\left|\varphi(x-\theta)-\varepsilon\sum_{k=1}^l\ZI_{(\theta-x_k,\theta+y_k)}(x)\right|\le \varepsilon.\label{1.2-5}
\emd
Without loss of generality we can suppose $0\le \theta<\pi$. Then we denote
\md0
k_0=\max\left\{ k: 0\le k \le l,\,\theta-x_k\le 0\right\}.
\emd
We define the desired intervals $I_j,\,j=1, 2, \dots, n=2l-k_0$, by
\md0
I_j=
\begin{cases}
(\theta-x_j,\theta+y_j) & \quad\text{if}\quad j\le k_0,\\
(0,\theta+y_j) & \quad\text{if}\quad k_0<j\le l,\\
(0,\theta-x_{j-l+k_0}] & \quad\text{if}\quad l<j\le n=2l-k_0.
\end{cases}
\emd
Using the equality
\md8
\ZI_{(\theta-x_k,\theta+y_k)}(x)
&=\ZI_{(0,\theta+y_k)}(x)-\ZI_{(0,\theta-x_k]}(x)\\
&=\ZI_{I_k}(x)-\ZI_{I_{k+l-k_0}}(x),\quad k_0<k\le l,
\emd
we get
\md1\label{1.2-6}
\varepsilon \sum_{k=1}^l\ZI_{(\theta-x_k,\theta+y_k)}(x)=\sum_{j=1}^n\varepsilon_j\ZI_{I_j}(x),
\emd
where
\md1\label{1.2-7}
\varepsilon_j=
\begin{cases}
\phantom{-}\varepsilon & \quad\text{if}\quad 1\le j\le l,\\
-\varepsilon & \quad\text{if}\quad l<j\le n.
\end{cases}
\emd
We note that $\varepsilon_j=-\varepsilon$ in the case when $I_j$ coincides with one of the intervals $(0,\theta-x_k]$, $k_0<k\le l$. Hence we have
\md1\label{1.2-8}
\sum_{j=l+1}^n|I_j|=\sum_{k=k_0+1}^l(\theta-x_k)\le l\cdot \theta\le \frac{\theta \|\varphi\|_\infty}{\varepsilon}.
\emd
From \e{1.2-5} and \e{1.2-6} we get
\md1\label{1.2-9}
\left|\varphi(x-\theta)-\sum_{j=1}^n\varepsilon_j\ZI_{I_j}(x)\right|\le \varepsilon
\emd
and therefore by \e{1.2-7} we obtain
\md0
\left|\int_\ZT \varphi(t)dt-\varepsilon \sum_{ j=1}^l |I_j|+\varepsilon \sum_{ j=l+1}^n |I_j|\right|\le 2\pi \varepsilon<2\pi.
\emd
This and \e{1.2-8} imply
\md0
\varepsilon\sum_{ j=1}^n |I_j|\le 2\varepsilon \sum_{ j=l+1}^n |I_j|+\|\varphi\|_1+2\pi  \le 2\theta \|\varphi\|_\infty +\|\varphi\|_1+2\pi,
\emd
which together with \e{1.2-3}, \e{1.2-4} and \e{1.2-9} completes the proof of lemma.
\end{proof}

We will use the following lemma in the proof of \trm{1.T3}.
\begin{lemma}\label{1.BV-means}
Let $\varphi\in\BV(\ZT)$ be a function of bounded variation and
\md0
\Delta_k = \bigcup_{j=0}^{n_k-1}\left[\frac{2\pi j}{n_k}-\delta_k,\frac{2\pi j}{n_k}+\delta_k\right] \subset\ZT,
\emd
where $n_k\in\ZN, \delta_k\in\ZT$ such that $n_k\to\infty$ as $k\to\infty$ and $0<\delta_k<\frac{\pi}{n_k},\,k=1,2,\dots$. Then
\md0
\lim_{k\to\infty} \frac{1}{|\Delta_k|}\int_{\Delta_k} \varphi(\theta+t)\,dt = \frac{1}{2\pi}\int_\ZT \varphi(t)\,dt,
\emd
where the convergence is uniform with respect to $\theta\in\ZT$.
\end{lemma}
\begin{proof}
Denote by $\Delta_k^j$ the $j$th component interval of $\Delta_k$ such that $\Delta_k=\cup_{0\le j<n_k}\Delta_k^j$. Condition $0<\delta_k<\frac{\pi}{n_k}$ implies that component intervals are pairwise disjoint and $|\Delta_k^j|=2\delta_k$. Let $\theta+\Delta_k^j = \{\theta+t\colon t\in\Delta_k^j\}$ and $\V{\varphi, [a,b]}$ be the total variation of function $\varphi$ on an interval $[a,b]\subset\ZT$. Then
\md9
&\left|\frac{1}{|\Delta_k|}\int_{\Delta_k} \varphi(\theta+t)\,dt - \frac{1}{2\pi}\int_\ZT \varphi(t)\,dt\right|\\
&\mspace{100mu}\le\left|\frac{1}{n_k}\sum_{j=0}^{n_k-1}\frac{1}{2\delta_k}\int_{\Delta_k^j}\varphi(\theta+t)\,dt
- \frac{1}{n_k}\sum_{j=0}^{n_k-1}\varphi\left(\theta+\frac{2\pi j}{n_k}\right)\right|\\
&\mspace{100mu} + \left|\frac{1}{n_k}\sum_{j=0}^{n_k-1}\varphi\left(\theta+\frac{2\pi j}{n_k}\right)
         - \frac{1}{2\pi}\int_\ZT \varphi(\theta+t)\,dt\right|\\
&\mspace{100mu}\le \frac{1}{n_k}\sum_{j=0}^{n_k-1}\frac{1}{2\delta_k}\int_{\Delta_k^j}\left|\varphi(\theta+t)-\varphi\left(\theta+\frac{2\pi j}{n_k}\right)\right|\,dt\\
&\mspace{100mu}+ \frac{1}{2\pi}\sum_{j=0}^{n_k-1}\int_{2\pi j/n_k}^{2\pi(j+1)/n_k}\left|\varphi(\theta+t)-\varphi\left(\theta+\frac{2\pi j}{n_k}\right)\right|\,dt\\
&\mspace{100mu}\le \frac{1}{n_k}\sum_{j=0}^{n_k-1}\V{\varphi, \theta+\Delta_k^j}
+ \frac{1}{n_k}\sum_{j=0}^{n_k-1} \V{\varphi, \left[\theta+\frac{2\pi j}{n_k}, \theta+\frac{2\pi(j+1)}{n_k}\right]}\\
&\mspace{100mu}\le \frac{2}{n_k}\V{\varphi, \ZT}.
\emd
The last term does not depend on $\theta$ and vanishes as $k\to\infty$, which completes the proof of the lemma.
\end{proof}

The next 3 lemmas are key ingredients of the proof of \trm{1.weakpp}.
\begin{lemma}\label{1.Phi*ineq1}
Let $\{\varphi_r\}$ be an arbitrary approximate identity and $\lambda(r)>0$ be any function. Then for any function $f\in L^1(\ZT)$
\md1
\sup_{\substack{|x-y|<\lambda(r)\\0<r<1}}\left|\int_{\lambda(r)\le|t|\le\pi} \varphi_r(t)f(y-t)\,dt \right| \le 8 C_\varphi\cdot Mf(x),\quad x\in\ZT.
\emd
\end{lemma}
\begin{proof}
Without loss of generality we may assume that $f$ is non-negative. Let $x, y\in\ZT,\,0<r<1$ such that $|x-y|<\lambda(r)$. We devide the interval $[\lambda(r),\pi]$ into $[2^{k-1}\lambda(r), 2^k\lambda(r)], k=1,2,\dots,Q=\lceil\log\frac{\pi}{\lambda(r)}\rceil$ and estimate the values of $\varphi_r(t)$ by its maximum in each divided interval:
\md2
\left|\int_{\lambda(r)}^{\pi} \varphi_r(t)f(y-t)\,dt\right|
&\le \sum_{k=1}^Q \int_{2^{k-1}\lambda(r)}^{2^k\lambda(r)} \varphi^*_r(t)f(y-t)\,dt\\
&\le \sum_{k=1}^Q \varphi^*_r\left(2^{k-1}\lambda(r)\right)\int_{2^{k-1}\lambda(r)}^{2^k\lambda(r)} f(y-t)\,dt\\
&\le \sum_{k=1}^Q \varphi^*_r\left(2^{k-1}\lambda(r)\right)\int_{\lambda(r)}^{2^k\lambda(r)} f(y-t)\,dt
\emd
Since $|x-y|<\lambda(r)$ we have
\md0
\int_{\lambda(r)}^{2^k\lambda(r)} f(y-t)\,dt \le \int_{0}^{(1+2^k)\lambda(r)} f(x-t)\,dt.
\emd
Therefore
\md2
\left|\int_{\lambda(r)}^{\pi} \varphi_r(t)f(y-t)\,dt\right|
&\le \sum_{k=1}^Q \varphi^*_r\left(2^{k-1}\lambda(r)\right)\int_{0}^{(1+2^k)\lambda(r)} f(x-t)\,dt\\
&\le Mf(x)\cdot \sum_{k=1}^Q \varphi^*_r\left(2^{k-1}\lambda(r)\right)(1+2^k)\lambda(r)\\
&\le 8 Mf(x)\cdot \sum_{k=0}^{Q-1} \varphi^*_r\left(2^k\lambda(r)\right)2^{k-1}\lambda(r)\\
&\le 8 Mf(x) \cdot \int_0^\pi \varphi^*_r(t)\,dt,
\emd
where in the last inequality we have used the following simple geometric inequlaity:
\md8
\varphi^*_r\left(\lambda(r)\right)\lambda(r) &+ \sum_{k=1}^{Q-1} \varphi^*_r\left(2^k\lambda(r)\right) 2^{k-1}\lambda(r)\\
&\le \int_0^{\lambda(r)}\varphi^*_r(t)\,dt + \sum_{k=1}^{Q-1}\int_{2^{k-1}\lambda(r)}^{2^k\lambda(r)}\varphi^*_r(t)\,dt \\
&\le \int_0^\pi \varphi^*_r(t)\,dt.
\emd
Thus we have
\md0
\left|\int_{\lambda(r)}^{\pi} \varphi_r(t)f(y-t)\,dt\right| \le 8 Mf(x)\cdot\int_0^\pi \varphi^*_r(t)\,dt.
\emd
In the same way we get
\md0
\left|\int_{-\pi}^{-\lambda(r)} \varphi_r(t)f(y-t)\,dt\right| \le 8 Mf(x)\cdot\int_{-\pi}^0 \varphi^*_r(t)\,dt.
\emd
Therefore
\md0
\sup_{\substack{|x-y|<\lambda(r)\\0<r<1}} \left|\int_{\lambda(r)\le|t|\le\pi} \varphi_r(t)f(y-t)\,dt\right|
\le 8 Mf(x)\cdot\sup_{0<r<1}\|\varphi^*_r\|_1 \le 8C_\varphi\cdot Mf(x).
\emd
\end{proof}

\begin{lemma}\label{1.lemma-TAbound}
Let $\{\varphi_r\}$ be an arbitrary approximate identity and $\mu(r),\,\lambda(r)$ are some functions with
\begin{enumerate}
\item $0<\mu(r)\le\lambda(r)\le\pi$,
\item $\lambda(r) \le C\mu(r)\FI^{-p}(r)$, for some $C>0$ and $p\ge 1$.
\end{enumerate}
Then for any $A\ge1$ and for any function $f\in L^1(\ZT)$
\md1
T_A f(x) \le \left(C\cdot\frac{M|f|^p(x)}{A}\right)^{1/p},\quad x\in\ZT,
\emd
where
\md4
T_A f(x) = \sup_{\substack{A\mu(r)<|x-y|<\lambda(r)\\0<r<1}}\FI(r)m_f(y,A\mu(r)),\\
m_f(y,t) = \frac{1}{2t}\int_{y-t}^{y+t}|f(u)|\,du.\nonumber
\emd
\end{lemma}
\begin{proof}
Without loss of generality we may assume that $f$ is non-negative. Using the definition of $T_A$ and Jensen\a s inequality we get
\md2
T_A^p f(x)
&= \sup_{\substack{A\mu(r)<|x-y|<\lambda(r)\\0<r<1}}\FI^p(r) m_f^p(y, A\mu(r))\\
&\le \sup_{\substack{A\mu(r)<|x-y|<\lambda(r)\\0<r<1}}\FI^p(r) m_{f^p}(y, A\mu(r))\\
&= \sup_{k\in\ZN} \sup_{\substack{2^{k-1}A\mu(r)<|x-y|\le 2^k A\mu(r)\\2^k A\mu(r)\le\lambda(r)\\0<r<1}} \FI^p(r) m_{f^p}(y, A\mu(r))
\emd
To estimate the inner supremum, first note that $2^k A\mu(r)\le\lambda(r)\le C\mu(r)\FI(r)^{-p}$ imples $\FI^p(r)\le C\left(2^k A\right)^{-1}$, where $C$ is the constant from condition $\mathit{2}$. Furthermore, since $2^{k-1}A\mu(r)<|x-y|\le2^k A\mu(r)$ we have
\md2
m_{f^p}(y,A\mu(r))
&= \frac{1}{2A\mu(r)} \int_{y-A\mu(r)}^{y+A\mu(r)}f^p(u)\,du\\
&\le \frac{1}{2A\mu(r)} \int_{x}^{x+(1+2^k)A\mu(r)}f^p(u)\,du\\
&\le \frac{(1+2^k)A\mu(r)}{2A\mu(r)} Mf^p(x) \le 2^k Mf^p(x).
\emd
Therefore
\md0
T_A^p f(x) \le \sup_{k\in\ZN} C\left(2^k A\right)^{-1}2^k Mf^p(x) = C\cdot\frac{Mf^p(x)}{A}.
\emd
\end{proof}

\begin{lemma}\label{1.lemma-supmulambda}
Let $\{\varphi_r\}$ be an arbitrary approximate identity and $\mu(r), \lambda(r)$ are some functions satisfying the conditions \textit{1.} and \textit{2.} from \lem{1.lemma-TAbound}. Then for any function $f\in L^1(\ZT)$
\md1\label{1.sup-mulambda}
\sup_{\substack{|x-y|<\lambda(r)\\0<r<1}}\left|\int_{\mu(r)\le|t|\le\lambda(r)} \varphi_r(t)f(y-t)\,dt \right|
\le \frac{4C^{1/p}}{2^{1/p}-1} \left(M|f|^p(x)\right)^{1/p},\quad x\in\ZT.
\emd
\end{lemma}
\begin{proof}
Again, we may assume that $f$ is non-negative. Let $x, y\in\ZT, 0<r<1$ and $|x-y|<\lambda(r)$. If $Q=\lceil\log\frac{\lambda(r)}{\mu(r)}\rceil$, we split the integral in \e{1.sup-mulambda} as follows
\md9\label{1.varphir-bound}
&\left|\int_{\mu(r)\le|t|\le\lambda(r)} \varphi_r(t)f(y-t)\,dt\right|\\
&\le \sum_{k=1}^Q \int_{2^{k-1}\mu(r)\le|t|\le2^k\mu(r)} \varphi^*_r(t)f(y-t)\,dt\\
&\le \sum_{k=1}^Q \max\left(\varphi^*_r\left(2^{k-1}\mu(r)\right), \varphi^*_r\left(-2^{k-1}\mu(r)\right)\right) \int_{|t|\le 2^k\mu(r)}f(y-t)\,dt\\
&= 2 \sum_{k=1}^Q 2^{k-1}\mu(r)\max\left(\varphi^*_r\left(2^{k-1}\mu(r)\right), \varphi^*_r\left(-2^{k-1}\mu(r)\right)\right) m_f(y,2^k\mu(r))\\
&\le 2\sum_{k=1}^Q \FI(r)m_f(y,2^k\mu(r)).
\emd
Then we split the domain of supremum in the followoing way:
\md9\label{1.sup-bound-1}
\sup_{\substack{|x-y|<\lambda(r)\\0<r<1}} \FI(r)m_f(y,A\mu(r))
&\le \sup_{\substack{|x-y|\le A\mu(r)\le\lambda(r)\\0<r<1}} \FI(r)m_f(y,A\mu(r))\\
&+   \sup_{\substack{A\mu(r)<|x-y|<\lambda(r)\\0<r<1}} \FI(r)m_f(y,A\mu(r)).
\emd
Notice that the second supremum is $T_A f(x)$. To estimate the first supremum, note that $|x-y|\le A\mu(r)\le\lambda(r)\le C\mu(r)\FI(r)^{-p}$ implies
\md1\label{1.varphir-bound-1}
\FI(r)\le  C^{1/p}A^{-1/p},
\emd
where $C$ is the constant from the condition $\mathit{2}$ of \lem{1.lemma-TAbound}. On the other hand, from $|x-y|\le A\mu(r)$ it follows $m_f(y,A\mu(r))\le Mf(x)$, which together with \e{1.varphir-bound-1}, \e{1.sup-bound-1} and \lem{1.lemma-TAbound} gives
\md9\label{1.term-bound}
\sup_{\substack{|x-y|<\lambda(r)\\0<r<1}} \FI(r)m_f(y,A\mu(r))
&\le C^{1/p}A^{-1/p} Mf(x) + T_A f(x)\\
&\le C^{1/p}A^{-1/p} (Mf^p(x))^{1/p} + \left(C\cdot\frac{Mf^p(x)}{A}\right)^{1/p}\\
&\le 2 C^{1/p}A^{-1/p} (Mf^p(x))^{1/p}.
\emd
Using \e{1.varphir-bound} and \e{1.term-bound} we get
\md2
&\sup_{\substack{|x-y|<\lambda(r)\\0<r<1}}\left|\int_{\mu(r)\le|t|\le\lambda(r)} \varphi_r(t)f(y-t)\,dt \right|\\
&\mspace{160mu}\le 2\sum_{k=1}^Q \sup_{\substack{|x-y|<\lambda(r)\\0<r<1}} \FI(r)m_f(y,2^k\mu(r))\\
&\mspace{160mu}\le 4C^{1/p}\sum_{k=1}^\infty 2^{-k/p} (Mf^p(x))^{1/p}\\
&\mspace{160mu}= \frac{4C^{1/p}}{2^{1/p}-1} (Mf^p(x))^{1/p},
\emd
which gives \e{1.sup-mulambda}.
\end{proof}

\begin{lemma}\label{1.rough-bound}
Let $\{\varphi_r\}$ be an arbitrary approximate identity and for some $1\le p<\infty$ the function $\lambda(r)$ satisfies
\md1\label{1.new-PIp-bound}
\sup_{0<r<1}\lambda(r)\|\varphi_r\|_q^p < \infty,
\emd
where $q=p/(p-1)$ is the conjugate index of $p$. Then for any function $f\in L^1(\ZT)$
\md1\label{1.sup-rough-bound}
\sup_{\substack{|x-y|<\lambda(r)\\0<r<1}}\left|\int_{|t|\le\lambda(r)} \varphi_r(t)f(y-t)\,dt \right|
\le C \left(M|f|^p(x)\right)^{1/p},\quad x\in\ZT,
\emd
where $C$ does not depend on function $f$.
\end{lemma}
\begin{proof}
The proof immediately follows from applying H\"{o}lder\a s inequality to the integral:
\md8
\sup_{\substack{|x-y|<\lambda(r)\\0<r<1}}\left|\int_{|t|\le\lambda(r)} \varphi_r(t)f(y-t)\,dt \right|
&\le \sup_{\substack{|x-y|<\lambda(r)\\0<r<1}}\|\varphi_r\|_q \cdot \left(\int_{|t|\le\lambda(r)} |f(y-t)|^p\,dt\right)^{1/p}\\
&\le \sup_{0<r<1}\|\varphi_r\|_q \cdot \left(\int_{|t|\le 2\lambda(r)} |f(x-t)|^p\,dt\right)^{1/p}\\
&\le \sup_{0<r<1}\|\varphi_r\|_q (4\lambda(r))^{1/p} \cdot \left(M|f|^p(x)\right)^{1/p},
\emd
which implies \e{1.sup-rough-bound} taking into account \e{1.new-PIp-bound}.
\end{proof}

\begin{lemma}\label{1.lemma-fatouC}
Let $\{\varphi_r\}$ be an arbitrary approximate identity and $\lambda:(0,1)\to(0,\infty)$ be a function with $\lambda(r)\to0$ as $r\to0$. Then for any continuous function $f\in C(\ZT)$
\md1\label{1.fatouC-conv}
\lim_{\stackrel{r\to 1}{y\in\lambda(r,x)}}\Phi_r(y,f) = f(x)
\emd
for any $x\in\ZT$.
\end{lemma}
\begin{proof}
Let $f\in C(\ZT)$ and $x\in\ZT$. Fix $\delta>0$ and $\theta\in\ZR$. Then we have
\md8
\left|\Phi_r(x+\theta,f) -f(x)\right|
&\le \left|\int_\ZT \varphi_r(t)\left[f(x+\theta-t)-f(x)\right]\,dt\right| + o(1)\\
&\le \int_{|t|<\delta} \varphi^*_r(t)\left|f(x+\theta-t)-f(x)\right|\,dt\\
&+   \int_{\delta\le|t|\le\pi} \varphi^*_r(t)\left|f(x+\theta-t)-f(x)\right|\,dt + o(1)\\
&\le \omega(f,\delta+\theta) \|\varphi^*_r\|_1 + 2\|f\|_C\cdot2\pi\varphi^*_r(\delta) + o(1)\\
&\le C_\varphi\cdot\omega(f,\delta+\theta) + 4\pi\|f\|_C\cdot\varphi^*_r(\delta) + o(1),
\emd
where $\omega(f,h)$ is the modulus of continuity in $C(\ZT)$ defined as
\md0
\omega(f,h) = \sup_{|x-y|\le h}|f(x)-f(y)|.
\emd
Therefore, from $\lambda(r)\to0$ as $r\to1$, we conclude
\md8
\limsup_{\stackrel{r\to 1}{|\theta|\le\lambda(r)}} \left|\Phi_r(x+\theta,f) -f(x)\right|
&\le \limsup_{r\to1}\big(C_\varphi\cdot\omega(f,\delta+\lambda(r)) + 4\pi\|f\|_C\cdot\varphi^*_r(\delta)\big)\\
&\le C_\varphi\cdot\omega(f,2\delta).
\emd
Since $\omega(f,h)\to0$ as $h\to0$ and $\delta$ can be taken arbitrarily small, we get \e{1.fatouC-conv}.
\end{proof}

\begin{lemma}\label{1.phi-r-bound}
If $\{\varphi_r\}$ is an arbitrary approximate identity, then for some $r_0\in(0,1)$
\md0
\frac{c}{\log\|\varphi_r\|_\infty} \le \FI(r) \le C_\varphi, \quad r_0<r<1,
\emd
where $c$ is a positive absolute constant.
\end{lemma}
\begin{proof}
Let $0<r<1$. Using the definitions of $\varphi^*_r(x)$ and $\FI(r)$ we conclude
\md0
\varphi_r(t) \le \varphi^*_r(t) \le \frac{\FI(r)}{|t|}, \quad t\in\ZT\setminus\{0\}.
\emd
Therefore, for a fixed $\delta\in(0,\pi)$ we have
\md8
1+o(1)
= \int_\ZT \varphi_r(t)\,dt
&\le \int_{|t|<\delta} \varphi^*_r(t)\,dt + \int_{\delta\le|t|\le\pi} \varphi^*_r(t)\,dt\\
&\le \|\varphi^*_r\|_\infty\int_{|t|<\delta}\,dt + \FI(r)\int_{\delta\le|t|\le\pi} \frac{dt}{|t|}\\
&= 2\delta\|\varphi_r\|_\infty + 2\FI(r)\log\frac{\pi}{\delta},
\emd
which implies
\md0
\FI(r) \ge \left(\frac{1}{2}+o(1)-\delta\|\varphi_r\|_\infty\right) \left(\log\frac{\pi}{\delta}\right)^{-1}.
\emd
Now, if we take $\delta = \pi/\|\varphi_r\|^2_\infty$, we get
\md0
\FI(r) \ge \left(\frac{1}{2}+o(1)-\frac{\pi}{\|\varphi_r\|_\infty}\right) \frac{1}{2\log\|\varphi_r\|_\infty},
\emd
which completes the proof of the first inequality (for example with $c=1/5$), since $\|\varphi_r\|_\infty\to\infty$ as $r\to1$. The second inequality can be deduced from the following:
\md0
\FI(r) = \sup_{x\in\ZT}|x\varphi^*_r(x)| \le \sup_{x\in\ZT}\left|\int_{|t|\le|x|}\varphi^*_r(t)\,dt\right| \le C_\varphi.
\emd
\end{proof}

If $f(x)$ is a function defined on a set $E\subset \ZT$ we denote
\md0
\OSC_{x\in E}f(x)=\sup_{x,y\in E}|f(x)-f(y)|.
\emd
\begin{lemma}\label{1.L1}
Let
\md0
U_n^\delta=\bigcup_{k=0}^{n-1}\left(\frac{\pi (2k+1-\delta)}{n},\frac{\pi (2k+1+\delta)}{n}\right),\quad n\in \ZN,\quad 0<\delta<\frac{1}{2},
\emd
and $J\subset \ZT$, $\pi >|J|\ge 16\pi/n$, is an arbitrary closed interval. If a measurable set $E\subset \ZT$ satisfies either
\md0
E\cap J=J\cap U_n^\delta \quad\text{or}\quad E\cap J=J\setminus U_n^\delta,
\emd
and  $\varphi\in L^\infty(\ZT)$ is an even decreasing on $[0,\pi]$ function, then
\md1\label{1.3-5}
\OSC_{\theta\in \left[x-\frac{4\pi}{n},x+\frac{4\pi}{n}\right]}\int_\ZT\varphi(\theta-t)\ZI_E(t)dt>\int_{-\frac{\pi\delta}{n}}^{\frac{\pi\delta}{n}}\varphi(t)dt-16\delta-2\pi\varphi\left(\frac{|J|}{4}\right),
\emd
for any $x\in J$.
\end{lemma}
\begin{proof}
We suppose $J=[a,b]$ and
\md0
\frac{2\pi (p-1)}{n}<a\le \frac{2\pi p}{n},\quad \frac{2\pi (q-1)}{n}<b\le \frac{2\pi q}{n}.
\emd
First we consider  the case
\md1\label{1.3-6}
E\cap J=J \cap U_n^\delta .
\emd
If $x\in J$, then
\md0
x\in I=\left[\frac{2\pi(m-1)}{n},\frac{2\pi m}{n} \right]
\emd
for some $p\le m\le q$. Without loss of generality we may assume that the center of $I$ is on the left hand side of the center of $J$. Then we will have
\md1\label{1.3-7}
b-\frac{ 2\pi(m+1)}{n}\ge \frac{|J|}{2}-\frac{4\pi}{n}\ge \frac{|J|}{4}.
\emd
It is clear, that the points
\md1\label{1.3-8}
\theta_1=\frac{ 2\pi m}{n}+\frac{\pi}{n},\quad \theta_2=\frac{ 2\pi(m+1)}{n}
\emd
are in the interval $[0,x+4\pi/n]$. Besides we have
\md8
\int_\ZT\varphi(\theta_1-t)\ZI_E(t)dt&-\int_\ZT\varphi(\theta_2-t)\ZI_E(t)dt\\
&=\int_{\theta_2-\pi}^a[\varphi(\theta_1-t)-\varphi(\theta_2-t)]\ZI_E(t)dt\\
&+\int_a^b[\varphi(\theta_1-t)-\varphi(\theta_2-t)]\ZI_E(t)dt\\
&+\int_b^{\theta_2+\pi}[\varphi(\theta_1-t)-\varphi(\theta_2-t)]\ZI_E(t)dt\\
&=A_1+A_2+A_3.
\emd
Since $\varphi$ is decreasing on $[0,\pi]$ we have
\md1\label{1.3-9}
A_1\ge 0.
\emd
If $t\in [b,\theta_2+\pi]$ then, using \e{1.3-7}, we get
\md0
t-\theta_2\ge b-\theta_2\ge \frac{|J|}{4},\quad t-\theta_1\ge  \frac{|J|}{4} ,
\emd
which implies
\md1\label{1.3-10}
|A_3|\le 2\pi \varphi\left(\frac{|J|}{4}\right).
\emd
To estimate $A_2$ we denote
\md1\label{1.3-11}
a_k=\int_{\pi( k-\delta)/n}^{\pi( k+\delta)/n}\varphi(t)dt,\quad k\in \ZZ.
\emd
We have
\md1\label{1.3-12}
a_0=\int_{-\pi\delta/n}^{\pi\delta/n}\varphi(t)dt.
\emd
Using properties of $\varphi $ we  have $a_k=a_{-k}$ and $a_1\ge a_2\ge\ldots $. Using Chebishev\a s inequality we have $\varphi(t)\le 1/t$. Thus we obtain
\md1\label{1.3-13}
a_k\le a_1=\int_{\pi( 1-\delta)/n}^{\pi( 1+\delta)/n}\varphi(t)dt\le \frac{2\pi\delta/n}{\pi(1-\delta)/n}=\frac{2\delta}{1-\delta}<4\delta,\quad k\ge 1.
\emd
Using \e{1.3-6}, \e{1.3-11}, \e{1.3-12} and \e{1.3-13}, we get
\md8
A_2&\ge \sum_{k=p}^{q-2}\int_{\pi (2k+1-\delta)/n}^{\pi (2k+1+\delta)/n}[\varphi(\theta_1-t)-\varphi(\theta_2-t)]dt-8\delta\\
&=\sum_{k=p}^{q-1}\int_{\pi(2(m-k)-\delta)/n}^{\pi(2(m-k)+\delta)/n}\varphi(t)dt
-\sum_{k=p}^{q-1}\int_{\pi(2(m-k)+1-\delta)/n}^{\pi(2(m-k)+1+\delta)/n}\varphi(t)dt-8\delta\\
&=\sum_{k=m-q+1}^{m-p}a_{2k}- \sum_{k=m-q+1}^{m-p}a_{2k+1}-8\delta\ge a_0-a_1-a_{-1}-8\delta\\
&>\int_{-\pi\delta/n}^{\pi\delta/n}\varphi(t)dt-16\delta.
\emd
Combining this with \e{1.3-9} and \e{1.3-10}, we get
\md0
\int_\ZT\varphi(\theta_1-t)\ZI_E(t)dt-\int_\ZT\varphi(\theta_2-t)\ZI_E(t)dt\ge \int_{-\frac{\pi\delta}{n}}^{\frac{\pi\delta}{n}}\varphi(t)dt-16\delta-2\pi\varphi\left(\frac{|J|}{4}\right),
\emd
which together with \e{1.3-8} implies \e{1.3-5}.  To deduce the case $E\cap J=J\setminus U_n^\delta$ notice, that for the complement $E^c$ we have
$E^c\cap J=J\cap U_n^\delta$ and so \e{1.3-5} holds for $E^c$. Therefore we obtain
\md8
\OSC_{\theta\in \left[x-\frac{4\pi}{n},x+\frac{4\pi}{n}\right]}&\int_\ZT\varphi(\theta-t)\ZI_E(t)dt\\
&=\OSC_{\theta\in \left[x-\frac{4\pi}{n},x+\frac{4\pi}{n}\right]}\left(\|\varphi \|_1-\int_\ZT\varphi(\theta-t)\ZI_E(t)dt\right)\\
&=\OSC_{\theta\in \left[x-\frac{4\pi}{n},x+\frac{4\pi}{n}\right]}\left(\int_\ZT\varphi(\theta-t)\ZI_{E^c}(t)dt\right)\\
&>\int_{-\frac{\pi\delta}{n}}^{\frac{\pi\delta}{n}}\varphi(t)dt-16\delta-2\pi\varphi\left(\frac{|J|}{4}\right),
\emd
which completes the proof of the lemma.
\end{proof}

%% file: sections/1.3.tex
\begin{proof}[Proof of \trm{1.T1}]
Without loss of generality we may assume that $x_0=0$ and $\mu'(x_0)=0$. We fix a function $\theta:(0,1)\to\ZR$ with $|\theta(r)|\le \lambda(r)$.  From $\PI(\lambda, \varphi)<\infty$ we get
\md1\label{1.2-10}
|\theta(r)|\cdot \|\varphi_r\|_\infty\le 2\PI,\quad r_0<r<1.
\emd
Using the property $\mathit{\Phi2}$, we may define a collection of numbers $\varepsilon_r>0$ such
that
\md1\label{1.2-11}
\varepsilon_r\searrow 0,\quad \delta_r=\sup\{|t|:\, \varphi_r(t)\ge\varepsilon_r \}\to 0\hbox { as } r\to 1.
\emd
Applying \lem{1.L1-}, for any $0<r<1$ we define a family of intervals $I_j^{(r)},\,j=1,2,\dots,n_r$ such that
\md5
|I_j^{(r)}|\le2\delta_r,\quad j=1,2,\ldots,n_r,\label{1.2-12}\\
\sum_{j=1}^{n_r}\left|I_j^{(r)}\right|<10\varepsilon_r^{-1}\max\{1,|\theta(r)|\cdot\|\varphi_r\|_\infty,\|\varphi_r \|_1\},\label{1.2-13}\\
\left|\varphi_r(\theta(r)-t)-\sum_{j=1}^{n_r}\varepsilon_j^{(r)} \ZI_{I_j^{(r)}}(t)\right|\le\varepsilon_r,\label{1.2-14}
\emd
where  $\varepsilon_j^{(r)}=\pm \varepsilon_r$. From \e{1.2-10} and \e{1.2-13} we conclude
\md1\label{1.2-15}
\varepsilon_r\cdot \sum_{j=1}^{n_r}\left|I_j^{(r)}\right|\le L,\quad r_0<r<1,
\emd
where $L$ is a positive constant. From \e{1.2-11} and \e{1.2-14} we obtain
\md1\label{1.2-16}
\Phi_r(\theta(r),d\mu)=\int_\ZT\varphi_r(\theta(r)-t)d\mu(t)=\sum_{j=1}^{n_r}\varepsilon_j^{(r)} \int_{I_j^{(r)}}d\mu(t)+o(1),
\emd
where $o(1)\to 0$  as $r\to 1$.  Using this, we get
\md1\label{1.2-16.1}
\left|\Phi_r(\theta(r),d\mu)\right|\le \varepsilon_r\cdot \sum_{j=1}^{n_r}\left|I_j^{(r)}\right|\cdot \frac{1}{\left|I_j^{(r)}\right|}\left|\int_{I_j^{(r)}}d\mu(t)\right|+o(1).
\emd
According to \e{1.2-11} and \e{1.2-12}, we have
\md0
\max_{1\le j\le n_r}\frac{1}{\left|I_j^{(r)}\right|}\left|\int_{I_j^{(r)}}d\mu(t)\right|\to \mu'(0)=0 \hbox{ as } r\to 1.
\emd
This together with \e{1.2-15} and \e{1.2-16} implies that $\Phi_r(\theta(r),d\mu)\to 0$ as $r\to 1$.
\end{proof}

\begin{proof}[Proof of \trm{1.T2}]
 Let $\theta(r)$ satisfies \e{1.2-10}. We again assume that $x_0=0$, $\mu'(x_0)=0$ and so we will have $|\mu|'(0)=0$. Then, repeating  the same process of the proof of \trm{1.T1} at this time for the functions $\varphi_r^*(t)$ together with the measure $|\mu|$, instead of \e{1.2-16} we obtain
\md0
\int_\ZT \varphi_r^*\left(\theta(r)-t\right)d|\mu|(t)=\sum_{j=1}^{n_r}\varepsilon_j^{(r)} \int_{I_j^{(r)}}d|\mu|(t)+o(1).
\emd
Then we get
\md8
\left|\Phi_r(\theta(r),d\mu)\right|&\le \int_\ZT \varphi_r^*(\theta(r)-t)d|\mu|(t)\\
&= \varepsilon_r\cdot \sum_{j=1}^{n_r}|I_j^{(r)}|\cdot \frac{1}{|I_j^{(r)}|}\int_{I_j^{(r)}}d|\mu|(t)+o(1).
\emd
Since $|\mu|$ is differentiable at $0$, we get
\md0
\Phi_r(\theta(r),d\mu)\to 0.
\emd
\end{proof}

% \begin{proof}[Proof of \trm{1.T2.1}]
%  Let $\theta(r)$ satisfies \e{1.2-10}. Taking into account the property $\mathit{\zK1}$, we again assume that $x_0=0$, $\mu'(x_0)=0$ and so we will have $|\mu|'(0)=0$. Then, repeating  the same process of the proof of \trm{1.T1} for the measure $|\mu|$, instead of \e{1.2-16.1} we get
% \md8
% \left|\mathcal{K}_r(\theta(r),d\mu)\right|
% &\le \int_\ZT \varphi_r(\theta(r)-t)d|\mu|(t)\\
% &= \varepsilon_r\cdot \sum_{j=1}^{n_r}|I_j^{(r)}|\cdot \frac{1}{|I_j^{(r)}|}\int_{I_j^{(r)}}d|\mu|(t)+o(1).
% \emd
% Since $|\mu|$ is differentiable at $0$, we get
% \md0
% \mathcal{K}_r(\theta(r),d\mu)\to 0.
% \emd
% \end{proof}

\begin{proof}[Proof of \trm{1.T3}]
For any $0<r<1$ there exist a point $x_r\in\ZT$, a number $0<\delta_r<\frac{1}{4}\lambda(r)$ and a measurable set $E_r\subset \ZT$ such that
\md3
E_r\subset (x_r-\delta_r,x_r+\delta_r),\quad |E_r|>\frac{3\delta_r}{2},\label{1.2-17}\\
|\varphi_r(x)|>\frac{\|\varphi_r\|_\infty}{2},\quad x\in E_r.\label{1.2-18}
\emd
From these relations it follows that $\varphi_r^*(x)>\frac{1}{2}\|\varphi_r\|_\infty$ if $x\in(-|x_r|,|x_r|)$. On the other hand, by property $\mathit{\Phi3}$ we have $\|\varphi_r^*\|_1\le C_\varphi$, which imples
\md1\label{1.2-19}
|x_r|\le \frac{2 C_\varphi}{\|\varphi_r\|_\infty},\quad 0<r<1.
\emd
Denote
\md5
n(r)=\left[\frac{4\pi}{\lambda(r)}\right]\in \ZN,\label{1.2-20}\\
\Delta_r=\bigcup_{k=0}^{n(r)-1}\left[\frac{2\pi k}{n(r)}-\delta_r,\frac{2\pi k}{n(r)}+\delta_r\right].\label{1.2-21}
\emd
If $x\in \ZT$ is an arbitrary point, then
\md0
x\in \left[\frac{2\pi k_0}{n(r)},\frac{2\pi (k_0+1)}{n(r)}\right)
\emd
for some $k_0\in\{0, 1, \dots, n(r)-1\}$. Consider the function
\md1\label{1.def-fr}
f_r(x)=\frac{\ZI_{\Delta_r}(x)}{|\Delta_r|} \sgn\varphi_r\left(\frac{2\pi k_0}{n(r)}+x_r-x\right).
\emd
Clearly $\|f_r\|_1=1$. Taking $\theta=x-x_r-\frac{2\pi k_0}{n(r)}$, from \e{1.2-19} and \e{1.2-20} we obtain
\md1\label{1.2-22}
|\theta |<\frac{2\pi}{n(r)}+|x_r| <\frac{2\pi\lambda(r)}{4\pi-\lambda(r)}+\frac{2 C_\varphi}{\|\varphi_r\|_\infty}\le\lambda(r)\left(\frac{1}{2}+\frac{2 C_\varphi}{\lambda(r)\|\varphi_r\|_\infty}\right).
\emd
Using the condition $\PI(\lambda,\varphi)=\infty$ and \lem{1.BV-means} we may fix a sequence $r_k\nearrow 1$ such that
\md5
\lambda(r_k)\|\varphi_{r_k}\|_\infty > C_\varphi\cdot 2^{k+3}\left(1 + k +\max_{1\le j<k}\frac{1}{|\Delta_{r_j}|}\right),\quad  k=1,2,\ldots ,\label{1.2-23}\\
\sup_{\theta\in\ZT}\frac{1}{|\Delta_{r_j}|}\int_{\Delta_{r_j}} \varphi^*_{r_k}(\theta-t)\,dt \le C_\varphi,\quad k=1, 2, \dots, j-1.\label{1.2-23.1}
\emd
From \e{1.2-22} and \e{1.2-23} we conclude
\md1\label{1.2-24}
|\theta |<\lambda(r),\hbox{ if } r=r_k.
\emd
Using \e{1.2-18}, \e{1.2-20} and \e{1.2-21}, for the same $x$ we get
\md9\label{1.2-25}
\Phi_r(x-\theta,f_r)
&=\int_{\ZT} \varphi_r\left(\frac{2\pi k_0}{n(r)}+x_r-t\right)f_r(t)dt\\
&=\frac{1}{|\Delta_r|} \int_{\Delta_r} \left|\varphi_r\left(\frac{2\pi k_0}{n(r)}+x_r-t\right)\right|\,dt\\
&\ge\frac{1}{2{\delta_r} n(r)}\int_{2\pi k_0/n(r)-{\delta_r}}^{2\pi k_0/n(r)+{\delta_r}} \left|\varphi_r\left(\frac{2\pi k_0}{n(r)}+x_r-t\right)\right|\,dt\\
&=\frac{1}{2{\delta_r} n(r)}\int_{x_r-\delta_r}^{x_r+\delta_r} \left|\varphi_r(u)\right|\,du\\
&\ge \frac{1}{2{\delta_r} n(r)}\cdot \frac{3\delta_r}{2}\cdot\frac{\|\varphi_r\|_\infty}{2}\ge \frac{3\lambda(r)\|\varphi_r\|_\infty}{16}.
\emd
Define
\md0
f(x)=\sum_{k=1}^\infty 2^{-k}f_{r_k}(x)\in L^1(\ZT),
\emd
and show that
\md0
\lim_{k\to\infty}\sup_{\theta\in\lambda(r_k,x)} \Phi_{r_k}(\theta,f) = \infty.
\emd
We split $\Phi_{r_k}(\theta,f)$ in the following way
\md9\label{1.Phi-split}
\Phi_{r_k}(\theta,f)
&= \sum_{j=1}^\infty 2^{-j}\Phi_{r_k}(\theta, f_{r_j})\\
&= \sum_{j=1}^{k-1} 2^{-j}\Phi_{r_k}(\theta, f_{r_j}) + 2^{-k}\Phi_{r_k}(\theta, f_{r_k}) + \sum_{j=k+1}^\infty 2^{-j}\Phi_{r_k}(\theta, f_{r_j})\\
&= S^1 + S^2 + S^3
\emd
From \e{1.2-23}, \e{1.2-24} and \e{1.2-25} it follows that
\md1\label{1.bound-S2}
\sup_{\theta\in\lambda(r_k,x)}S^2 = \sup_{\theta\in\lambda(r_k,x)}2^{-k}\Phi_{r_k}(\theta,f_{r_k})
\ge C_\varphi\left(1 + k +\max_{1\le j<k}\frac{1}{|\Delta_{r_j}|}\right)
\emd
Furthermore, using \e{1.def-fr} and $\mathit{\Phi3}$ propery of $\{\varphi_r\}$ we get
\md9\label{1.bound-S1}
\sup_{\theta\in\lambda(r_k,x)} |S^1|
&= \sup_{\theta\in\lambda(r_k,x)} \left|\sum_{j=1}^{k-1} 2^{-j}\int_\ZT \varphi_{r_k}(\theta-t)f_{r_j}(t)\,dt\right|\\
&\le \sup_{\theta\in\lambda(r_k,x)} \sum_{j=1}^{k-1} \frac{2^{-j}}{|\Delta_{r_j}|} \int_{\Delta_{r_j}}|\varphi_{r_k}(\theta-t)|\,dt\\
&\le \sum_{j=1}^{k-1} \frac{2^{-j}}{|\Delta_{r_j}|}\int_\ZT\varphi^*_{r_k}(u)\,du
 \le C_\varphi\cdot\max_{1\le j<k}\frac{1}{|\Delta_{r_j}|}.
\emd
Finally, using \e{1.2-23.1} we get
\md9\label{1.bound-S3}
\sup_{\theta\in\lambda(r_k,x)} |S^3|
&= \sup_{\theta\in\lambda(r_k,x)} \left|\sum_{j=k+1}^\infty 2^{-j}\int_\ZT \varphi_{r_k}(\theta-t)f_{r_j}(t)\,dt\right|\\
&\le \sup_{\theta\in\lambda(r_k,x)} \sum_{j=k+1}^\infty 2^{-j}\cdot\frac{1}{|\Delta_{r_j}|} \int_{\Delta_{r_j}}|\varphi_{r_k}(\theta-t)|\,dt\\
&\le \sum_{j=k+1}^\infty 2^{-j}\cdot\sup_{\theta\in\ZT}\frac{1}{|\Delta_{r_j}|} \int_{\Delta_{r_j}}\varphi^*_{r_k}(\theta-t)\,dt \le C_\varphi
\emd
So, from \e{1.bound-S1}, \e{1.bound-S2}, \e{1.bound-S3} and \e{1.Phi-split} it follows
\md0
\sup_{\theta\in\lambda(r_k,x)} \Phi_{r_k}(\theta,f)
\ge \sup_{\theta\in\lambda(r_k,x)} S^2 - \sup_{\theta\in\lambda(r_k,x)} |S^1| - \sup_{\theta\in\lambda(r_k,x)} |S^3|
\ge C_\varphi \cdot k,
\emd
which imples \e{1.2-2}.
\end{proof}

\begin{proof}[Proof of \trm{1.weakpp}]
Without loss of generality we may assume that $f$ is non-negative.
Furthermore, we may assume that $\TPI_p\ge C_\varphi^p$ and $\lambda(r)\|\varphi_r\|_\infty\varphi_*^{p-1}(r)\ge C_\varphi^p$ for all $r\in(0,1)$. Otherwise, instead of $\lambda(r)$ we would define a new $\bar{\lambda}(r)$ as
\begin{equation*}
\bar{\lambda}(r) \deff \frac{\max\left(\TPI_p, C_\varphi^p\right)}{\|\varphi_r\|_\infty\varphi_*^{p-1}(r)} \ge \lambda(r), \quad 0<r<1,
\end{equation*}
for which those assumptions would hold. Denote $\mu(r)=\varphi_*(r)/\|\varphi_r\|_\infty$ and notice that
\begin{equation*}
\lambda(r) \ge \frac{C_\varphi^p}{\|\varphi_r\|_\infty\varphi_*^{p-1}(r)} \ge \frac{\varphi_*(r)}{\|\varphi_r\|_\infty}=\mu(r), \quad 0<r<1.
\end{equation*}

Let $x,y\in\ZT,\,0<r<1$ and $|x-y|<\lambda(r)$. We split the integral $\Phi_r(y,f)$ as follows
\md9\label{1.I123}
\Phi_r(y,f)
&= \int_\ZT \varphi_r(t)f(y-t)\,dt\\
&= \int_{|t|\le\mu(r)} \varphi_r(t)f(y-t)\,dt\\
&+ \int_{\mu(r)<|t|<\lambda(r)} \varphi_r(t)f(y-t)\,dt\\
&+ \int_{\lambda(r)\le|t|\le\pi} \varphi_r(t)f(y-t)\,dt = I^1 + I^2 + I^3.
\emd
First of all, from \lem{1.Phi*ineq1} we have
\md1\label{1.I3ineq}
\sup_{\substack{|x-y|<\lambda(r)\\0<r<1}} |I^3| \le 8 C_\varphi\cdot Mf(x).
\emd
Notice that from the condition $\TPI_p(\lambda, \varphi)<\infty$ it follows that
\md0
\lambda(r)\le \TPI_p\cdot\mu(r)\FI^{-p}(r).
\emd
Hence, from \lem{1.lemma-supmulambda} we get
\md1\label{1.I2ineq}
\sup_{\substack{|x-y|<\lambda(r)\\0<r<1}} |I^2| \le \frac{4\TPI_p^{1/p}}{2^{1/p}-1} \left(Mf^p(x)\right)^{1/p}.
\emd
Furthermore, using the definition of $\mu(r)$, for $I_1$ we obtain
\md8
|I^1|
&\le \int_{|t|\le\mu(r)} \varphi^*_r(t)f(y-t)\\
&\le \|\varphi_r\|_\infty \int_{-\mu(r)}^{\mu(r)}f(y-t)\,dt\\
&=  2\mu(r)\|\varphi_r\|_\infty m_f(y,\mu(r)) =  2\FI(r)m_f(y,\mu(r)),
\emd
where
\md0
m_f(y,t) = \frac{1}{2t}\int_{y-t}^{y+t}|f(u)|\,du, \quad y\in\ZT,\,t>0.
\emd
To estimate $I_1$ we split the supremum into two parts as we did in \lem{1.lemma-supmulambda}:
\md9\label{1.I1ineq}
\sup_{\substack{|x-y|\le\mu(r)\le\lambda(r)\\0<r<1}} I^1
&\le \sup_{\substack{|x-y|\le\mu(r)\le\lambda(r)\\0<r<1}} 2\FI(r)m_f(y,\mu(r))\\
&+ \sup_{\substack{\mu(r)<|x-y|<\lambda(r)\\0<r<1}} 2\FI(r)m_f(y,\mu(r))
\emd
Notice that the second supremum is $T_1 f(x)$, which can be estimated due to \lem{1.lemma-TAbound}. To estimate the first one, note that $\mu(r)\le\lambda(r)\le \TPI_p\mu(r)\FI(r)^{-p}$ implies
\md1\label{1.varphi-ineq}
\FI(r)\le\TPI_p^{1/p}.
\emd
On the other hand, from $|x-y|\le\mu(r)$ implies $m_f(y,\mu(r))\le Mf(x)$, which together with \e{1.I1ineq}, \e{1.varphi-ineq} and \lem{1.lemma-TAbound} gives
\md9\label{1.I1ineq-1}
\sup_{\substack{|x-y|\le\mu(r)\le\lambda(r)\\0<r<1}} |I^1| \le 2\TPI_p^{1/p} \left(Mf^p(x)\right)^{1/p}.
\emd
Then, combining \e{1.I3ineq}, \e{1.I2ineq}, \e{1.I1ineq-1} and \e{1.I123}, we get
\md0
\sup_{\substack{|x-y|<\lambda(r)\\0<r<1}} \Phi_r(y,f)
\le \left( 2\TPI_p^{1/p} + \frac{4\TPI_p^{1/p}}{2^{1/p}-1} + 8 C_\varphi\right) \left(Mf^p(x)\right)^{1/p},
\emd
which implies \e{1.Phi*-bound}.

To get weak type inequality for $\Phi^*_\lambda$, note that
\md2
\left|\left\{x\in\ZT \colon \Phi_{\lambda}^*(x,f) > t\right\}\right|
&= \left|\left\{x\in\ZT \colon \left(\Phi_{\lambda}^*(x,f)\right)^p > t^p\right\}\right|\\
&\le \left|\left\{x\in\ZT \colon Mf^p(x) > t^p/C^p\right\}\right|\\
&\le \frac{C_MC^p}{t^p} \|f^p\|_1 = \frac{C_MC^p}{t^p} \|f\|_p^p,
\emd
where $C_M=\|M\|_{L^1\to L^{1,w}}$ is the weak $(1,1)$ norm of the maximal operator $M$.
\end{proof}

%% file: sections/1.5.tex
\begin{proof}[Proof of \trm{1.T5}]
Since $\PI_\infty=0$, then for any $0<\varepsilon <1/2$ we may chose $\delta>0$ and $0<\tau<1$, such that
\md1\label{1.3-1}
\int_{-\delta \lambda(r)}^{\delta \lambda(r)}\varphi_r(t)dt<\varepsilon,\quad \tau<r<1.
\emd
Then we define
\md2
\varphi^{(1)}_r(x)=\left\{
\begin{array}{lrl}
\varphi_r(x)-\varphi_r(\delta \lambda(r))& \hbox { if }&  |x|\le \delta \lambda(r),\\
0&\hbox { if }& \delta \lambda(r)<|x|<\pi.
\end{array}
\right.
\emd
and
\md0
\varphi^{(2)}_r(x)=\frac{\varphi_r(x)-\varphi^{(1)}_r(x)}{\|\varphi_r(x)-\varphi^{(1)}_r\|_{L^1}}
=\frac{\varphi_r(x)-\varphi^{(1)}_r(x)}{1-l_r}
\emd
where
\md0
l_r=\int_{-\delta \lambda(r)}^{\delta \lambda(r)}\left(\varphi_r(t)-\varphi_r(\delta \lambda(r)\right)dt<\varepsilon<\frac{1}{2},\quad \tau<r<1.
\emd
It is clear, that $\{\varphi^{(2)}_r\}$ is a regular approximate identity and we have
\md1\label{1.3-2}
\varphi_r(x)=\varphi_r^{(1)}(x)+(1-l_r)\varphi_r^{(2)}(x).
\emd
From \e{1.3-1} it follows that
\md1\label{1.3-3}
\left|\int_\ZT\varphi^{(1)}_r(x-t)f(t)dt\right|\le \|f\|_\infty\int_{-\delta \lambda(r)}^{\delta \lambda(r)}\varphi_r(t)dt\le \varepsilon\|f\|_\infty
\emd
and
\md0
\varphi_r(\delta \lambda(r))\cdot 2\delta \lambda(r)<\varepsilon,\quad \tau<r<1.
\emd
Thus, using the definition of $\varphi^{(2)}_r(x)$, we get
\md0
\|\varphi_r^{(2)}\|_\infty\cdot  \lambda(r)<\frac{\varepsilon}{2\delta(1-l_r)}<\frac{\varepsilon}{4\delta}
\emd
Using this and \trm{1.T1} we conclude, that
\md1\label{1.3-4}
\lim_{\stackrel{r\to 1}{y\in\lambda(r,x)}}\int_\ZT\varphi^{(2)}_r(y-t)f(t)dt=f(x)
\emd
at any Lebesgue point.
Now without loss of generality we assume that $f(x)\ge 0$. If $x$ is an arbitrary Lebesgue point, using \e{1.3-2}, \e{1.3-3} and \e{1.3-4} we get
\md2
&\limsup_{\stackrel{r\to 1}{y\in\lambda(r,x)}}\Phi_r(y,f)\le \varepsilon \|f\|_\infty+f(x),\\
&\liminf_{\stackrel{r\to 1}{y\in\lambda(r,x)}}\Phi_r(y,f)\ge -\varepsilon \|f\|_\infty+(1-\varepsilon)f(x).
\emd
Since $\varepsilon $ can be taken sufficiently small, we get
\md0
\lim_{\stackrel{r\to 1}{y\in\lambda(r,x)}}\Phi_r(y,f)=f(x),
\emd
and the theorem is proved.
\end{proof}

\begin{proof}[Proof of \trm{1.T6}]
Since $\PI_\infty>0$, there exist sequences $\delta_k\searrow 0$ and $r_k\to 1$, such that
\md1\label{1.3-0}
\int_{-\delta_k \lambda(r_k)}^{\delta_k \lambda(r_k)}\varphi_{r_k}(t)dt>\frac{\PI_\infty}{2}, \quad k=1,2,\ldots .
\emd
Denote
\md1\label{1.3-14}
U_k=U_{n_k}^{\delta_k},\quad n_k=\left[\frac{\pi}{\lambda(r_k)}\right],
\emd
where $U_n^\delta$ is defined in the \lem{1.L1}.
Define the sequences of measurable sets $E_n$ by
\md0
E_1=U_1,\quad E_k=E_{k-1}\bigtriangleup U_k=(E_{k-1}\setminus U_k)\cup(U_k\setminus E_{k-1}),\quad k>1
\emd
We say $J$ is an adjacent interval for $E_k$, if it is a maximal interval containing either in $E_k$ or $(E_k)^c$. The family of all this intervals form a covering of whole $\ZT$.
It is easy to observe, that a suitable selection of  $\delta_k$ and $r_k$ may provide
\md3
&\varphi_{r_k}\left(\frac{|J|}{4}\right)<\frac{\PI_\infty}{16\pi},\hbox { if } J \hbox{ is adjacent for } E_{k-1},\label{1.3-15}\\
&\delta_j\le \frac{\PI_\infty}{2^{j+5}\|\varphi_{r_k}\|_\infty},\quad j\ge k+1,\label{1.3-16}
\emd
It is easy to observe, that if $k<m$, then
\md1\label{1.3-17}
\|\ZI_{E_k}-\ZI_{E_m}\|_1=|E_k\bigtriangleup E_m|\le \sum_{j\ge k+1}|U_j|
\emd
This implies, that $\ZI_{E_n}$ converges to a function $f\in L^1$. Using Egorov\a s theorem, we conclude that $f=\ZI_E$ for some measurable set $E\subset \ZT$. Tending  $m$ to infinity,  from \e{1.3-16} and \e{1.3-17} we get
\md1\label{1.3-18}
|E_k\bigtriangleup E|\le\left|\bigcup_{j\ge k+1} U_j\right|\le 2\pi\sum_{j\ge k+1}\delta_j\le \frac{\PI_\infty}{16\|\varphi_{r_k}\|_\infty}.
\emd
Fix a point $x\in\ZT$. We have $x\in J$ where $J$ is an adjacent interval for $E_{k-1}$.
From the definition of $E_k$ it follows that either
\md0
E_k\cap J=J\cap U_k \quad\text{or}\quad E_k\cap J=J\setminus U_k.
\emd
From \e{1.3-14} we have
\md0
\lambda(r_k,x)=(x-\lambda(r_k),x+\lambda(r_k))\subset \left[x-\frac{4\pi}{n_k},x+\frac{4\pi}{n_k}\right].
\emd
Thus, applying \lem{1.L1}, \e{1.3-0} and \e{1.3-15}, we get
\md8
\OSC_{\theta\in \lambda(r_k,x)}&\Phi_{r_k}(\theta,\ZI_{E_k})\\
&\ge\OSC_{\theta\in \left[x-\frac{4\pi}{n_k},x+\frac{4\pi}{n_k}\right]}\Phi_{r_k}(\theta,\ZI_{E_k})\\
&\ge\int_{-\frac{\pi\delta_k}{n_k} }^{\frac{\pi\delta_k}{n_k}}\varphi_{r_k}(t)dt-16\delta_k-2\pi\varphi_{r_k}\left(\frac{|J|}{4}\right)\\
&\ge\int_{-\delta_k \lambda(r_k)}^{\delta_k \lambda(r_k)}\varphi_{r_k}(t)dt-16\delta_k-\frac{\PI_\infty}{8}\\
&\ge\frac{\PI_\infty}{4}-16\delta_k,
\emd
where
\md0
\OSC_{x\in E}f(x)=\sup_{x,y\in E}|f(x)-f(y)|.
\emd

From \e{1.3-18} we conclude
\md8
\OSC_{\theta\in \lambda(r_k,x)}&\Phi_{r_k}(\theta,\ZI_E)\\
&>\OSC_{\theta\in \lambda(r_k,x)}\Phi_{r_k}(\theta,\ZI_{E_k})-\frac{\PI_\infty}{16}\ge \frac{\PI_\infty}{8}-16\delta_k,
\emd
which completes the proof of the theorem since $\delta_k\to 0$.
\end{proof}

%% file: sections/2.1.tex
In this chapter we generalize Littlewood\a s theorem for the integrals with general kernels. Here we remind Littlewood\a s theorem as well as generalized versions for Blaschke products and harmonic functions.

\begin{customtheorem}{\ref{OldLittlewood}}[Littlewood, 1927]
If a continuous function $\lambda(r):[0,1]\to \ZR$ satisfies the conditions
\md1\label{2.0-1}
\lambda(1)=0,\quad  \lim_{r\to 1}\frac{\lambda(r)}{1-r} = \infty,
\emd
then there exists a bounded analytic function $f(z)$, $z\in D$, such that the boundary limit
\md0
\lim_{r\to 1}f\left(re^{i\left(x+\lambda(r)\right)}\right)
\emd
does not exist almost everywhere on $\ZT$.
\end{customtheorem}

\begin{customtheorem}{\ref{2.OT1}}[Lohwater and Piranian, 1957]
If a continuous function $\lambda(r)$  satisfies \e{2.0-1}, then there exists a Blaschke product $B(z)$ such that the limit
\md0
\lim_{r\to 1}B\left(re^{i\left(x+\lambda(r)\right)}\right)
\emd
does not exist for any $x\in \ZT$.
\end{customtheorem}

\begin{customtheorem}{\ref{OldAikawa}}[Aikawa, 1990]
If $\lambda(r)$ is continuous and satisfies the condition \e{2.0-1},
then there exists a bounded harmonic function $u(z)$ on the unit disc, such that the limit
\md0
\lim_{r\to 1} u\left(re^{i\left(x+\lambda(r)\right)}\right)
\emd
does not exist for any $x\in \ZT$.
\end{customtheorem}

In \sect{littlewood-IE} we construct a characteristic function with Littlewood type divergence property for general kernels:
\begin{theorem}\label{2.T1}
Let $\{\varphi_r\}$ be a family of kernels with $\mathit{\Phi1,\,\Phi4,\,\Phi5}$. If a function $\lambda\in C[0,1]$ satisfies the conditions $\lambda(1)= 0$ and
\md0
\PI^*(\lambda,\varphi) = \limsup_{\delta \to 0}\liminf_{r\to 1}\int_{-\delta \lambda(r)}^{\delta \lambda (r)}\varphi_r(t)dt>\frac{1}{2},
\emd
then there exists a measurable set $E\subset \ZT$ such that
\md0
\limsup_{r\to 1}\Phi_r\left(x+\lambda(r),\ZI_E\right)- \liminf_{r\to 1}\Phi_r\left(x+\lambda(r),\ZI_E\right)\ge2\PI^*-1.
\emd
\end{theorem}

In \sect{littlewood-blaschke} we construct a Blaschke product with Littlewood type divergence property for general kernels:
\begin{theorem}\label{2.T2}
Let a family of kernels $\{\varphi_r\}$ satisfies $\mathit{\Phi1,\,\Phi4,\,\Phi5}$ and for $\lambda\in C[0,1]$ we have $\lambda(1)= 0$ and $\PI^*(\lambda,\varphi)=1$. Then there exists a function $B\in L^\infty(\ZT)$, which is the boundary function of a Blaschke product, such that the limit
\md0
\lim_{r\to 1}\Phi_r\left(x+\lambda(r),B\right)
\emd
does not exist for any $x\in\ZT$.
\end{theorem}

%% file: sections/2.2.tex
We consider the sets
\md1\label{2.p4}
U(n,\delta) = \bigcup_{j=0}^{n-1}\left(\frac{\pi (2j-\delta)}{n},\frac{\pi (2j+\delta)}{n}\right)\subset\ZT,
\emd
which will be used in the proofs of both theorems.

\begin{proof}[Proof of \trm{2.T1}]
Using the definition of $\PI^*$ and  the absolute continuity property $\mathit{\Phi5}$, we may choose numbers  $\delta_k$, $u_k$, $v_k\,(k\in\ZN)$, satisfying
\md5
\delta_k<2^{-k-5},\quad 1>v_k>u_k\to 1,\quad 3\lambda(v_k)\le \lambda(u_k)<\pi,\label{2.p9}\\
\int_{-\delta_k \lambda(u_k)}^{\delta_k \lambda(u_k)}\varphi_{u_k}(t)dt>\PI^*\cdot(1-2^{-k}), \quad k=1,2,\ldots ,\label{2.p8}\\
\int_e|\varphi_r(t)|dt<2^{-k},\label{2.p10}
\emd
where the last bound holds whenever
\md1\label{2.p31}
0<r<v_k,\quad |e|\le 10\pi \sum_{j\ge k+1}\sqrt[4]{\delta_j}.
\emd
We will consider the same sequences \e{2.p9} with properties \e{2.p8}$-$\e{2.p31} in the proof of \trm{2.T2} as well. We note that $\sqrt[4]{\delta_j}$ in \e{2.p31} is necessary only in the proof of \trm{2.T2}, but for \trm{2.T1}  just $\delta_j$ is enough.
Denote
\md1\label{2.c5}
U_k=U(n_k,5\delta_k),\quad n_k=\left[\frac{5\pi}{\lambda(u_k)}\right],\quad k\in \ZN,
\emd
and define the sequence of measurable sets $E_k\subset \ZT$ by
\md3
&E_1=U_1,\\
&E_k=\left\{
\begin{array}{lc}
E_{k-1}\setminus U_k &\hbox{ if }k \hbox{ is even},\\
E_{k-1}\cup U_k&\hbox{ if }k \hbox{ is odd}.
\end{array}\label{2.c6}
\right.
\emd
It is easy to observe, that if $k<m$, then
\md1\label{2.c17}
\|\ZI_{E_k}-\ZI_{E_m}\|_1=|E_k\bigtriangleup E_m|\le \sum_{j\ge k+1}|U_j|.
\emd
This implies that $\ZI_{E_n}$ converges to a function $f$ in $L^1$ norm. Using Egorov\a s theorem, we conclude that $f=\ZI_E$ for some measurable set $E\subset \ZT$. Tending  $m$ to infinity,  from \e{2.c17} we get
\md1\label{2.c11}
|E\bigtriangleup E_k|=|(E\setminus E_k)\cup(E_k\setminus E)|\le \sum_{j\ge k+1}|U_j|\le  10\pi \sum_{j\ge k+1}\delta_j.
\emd
Take an arbitrary $x\in \ZT$. There exists an integer  $1\le j_0\le n_k$ such that
\md0
\frac{2\pi j_0}{n_k}-x\in \left[\frac{2\pi }{n_k},\frac{4\pi }{n_k}\right]\subset \left[\frac{\lambda(u_k)}{3},\lambda(u_k)\right]\subset [\lambda(v_k),\lambda(u_k)]
\emd
and therefore, since $\lambda(r)$ is continuous, we may find a number $r$, $u_k\le r\le v_k$, such that
\md1\label{2.g1}
\lambda(r)=\frac{2\pi j_0}{n_k}-x.
\emd
If  $k\in \ZN$ is odd, then according to the definition of $E_k$ we get
\md0
E_k\supset U_k\supset I=\left(\frac{\pi(2 j_0+5\delta_k)}{n_k},\frac{\pi(2 j_0-5\delta_k)}{n_k}\right).
\emd
Thus, using \e{2.p8}, \e{2.g1} as well as the definition of $n_k$ from \e{2.c5}, we conclude
\md9\label{2.g15}
\Phi_r(x+\lambda(r),\ZI_{E_k})&\ge \int_I\varphi_r(x+\lambda(r)-t)dt\\
&= \int_I\varphi_{r}\left(\frac{2\pi j_0}{n_k}-t\right)dt\\
&=\int_{-5\pi \delta_k/n_k}^{5\pi \delta_k/n_k}\varphi_{r}\left(t\right)dt\\
&\ge \int_{-\delta_k \lambda(u_k)}^{\delta_k \lambda(u_k)}\varphi_{r}(t)dt>\PI^*\cdot(1-2^{-k}).
\emd
From \e{2.p10} and \e{2.c11} it follows that
\md0
\left|\Phi_r\left(t,\ZI_E\right)-\Phi_{r}\left(t,\ZI_{E_k}\right)\right|<2^{-k},\quad t\in\ZT,\quad 0<r<v_k,
\emd
and hence from \e{2.g15} we obtain
\md1\label{2.c8}
\limsup_{r\to 1}\Phi_r\left(x+\lambda(r),\ZI_E\right)\ge \PI^*.
\emd
If $k\in\ZN$ is even, then we have $E_k\cap U_k=\varnothing$ and therefore $E_k\cap I=\varnothing$.
Thus we get
\md2
\Phi_r(x+\lambda(r),\ZI_{E_k})&\le  \int_\ZT\varphi_r(x+\lambda(r)-t)dt-\int_I\varphi_r(x+\lambda(r)-t)dt\\
&\le 1- \int_{-\delta_k \lambda(u_k)}^{\delta_k \lambda(u_k)}\varphi_r\left(t\right)dt\le 1-\PI^*(1-2^{-k})
\emd
and similarly we get
\md1\label{2.c9}
\liminf_{r\to 1}\Phi_r\left(x+\lambda(r),\ZI_E\right)\le1-\PI^*.
\emd
Relations \e{2.c8} and \e{2.c9} complete the proof of the theorem.
\end{proof}

%% file: sections/2.3.tex
The following finite Blaschke products
\md1\label{2.p1}
b(n,\delta,z)=\frac{z^n-\rho^n}{\rho^nz^n-1}=\prod_{k=0}^{n-1}\frac{z-\rho  e^{\frac{2\pi ik}{n}}}{\rho  e^{\frac{2\pi ik}{n}}z-1},\quad \rho=e^{-\sqrt{\delta}/n}.
\emd
play significant role in the proof of \trm{2.T2}. Similar products are used in the proof of theorem \otrm{2.OT1} too. If $z=e^{ix}$, then \e{2.p1} defines a continuous function in $H^\infty(\ZT)$. We will use the set $U(n,\delta)$ defined in \e{2.p4}. The following lemma shows that on $U(n,\delta)$ the function \e{2.p1} is approximative $-1$, and outside of $U(n,\sqrt[4]{\delta})$ is approximative $1$.
\begin{lemma}\label{2.L1}
There exists an absolute constant $C>0$ such that
\md3
&\left|b\left(n,\delta,e^{ix}\right)+1\right|\le C\sqrt{\delta},\quad x\in U(n,\delta),\label{2.p5}\\
&\left|b\left(n,\delta,e^{ix}\right)-1\right|\le C\sqrt[4]{\delta},\quad x\in \ZT\setminus U(n,\sqrt[4]{\delta}).\label{2.p23}
\emd
\end{lemma}
\begin{proof}
Deduction of these inequalities based on the inequalities
\md0
\frac{|x|}{2}\le |e^{i x}-1|\le 2|x|, \quad\text{if}\quad |x|\le\pi.
\emd
If  $x\in U(n,\delta)$, then we have
\md9\label{2.p5-}
\left|b\left(n,\delta,e^{ix}\right)+1\right|&=\left|\frac{(e^{inx}-1)(\rho^n+1)}{\rho^ne^{inx}-1}\right|\le \frac{4\pi \delta}{1-e^{-\sqrt{\delta}}},\\
&\le \frac{4e\pi \delta}{e^{\sqrt{\delta}}-1}\le\frac{8e\pi \delta}{\sqrt{\delta}}\le C\sqrt{\delta}.
\emd
If $x\in \ZT\setminus U(n,\sqrt[4]{\delta})$, then $e^{inx}=e^{i\alpha}$ with $\pi\sqrt[4]{\delta}<|\alpha|<\pi $. Thus we obtain
\md9\label{2.p23-}
\left|b\left(n,\delta,e^{ix}\right)-1\right|&= \left|\frac{(e^{inx}+1)(1-\rho^n)}{\rho^ne^{inx}-1}\right|= \frac{2(e^{\sqrt{\delta}}-1)}{|e^{inx}-e^{\sqrt{\delta}}|}\\
&\le \frac{4 \sqrt{\delta}}{|e^{inx}-1|-|e^{\sqrt{\delta}}-1|}\le\frac{4 \sqrt{\delta}}{\pi\sqrt[4]{\delta}/2-2\sqrt{\delta}}\le C\sqrt[4]{\delta}.
\emd
\end{proof}
\begin{proof}[Proof of \trm{2.T2}]
First we choose numbers  $\delta_k$, $u_k$, $v_k$ $(k\in\ZN)$, satisfying \e{2.p9}$-$\e{2.p10} with $\PI^*=1$. Then we
denote
\md1\label{2.p22}
b_k(x)=b(n_k,\delta_k,e^{ix}), \quad n_k=\left[\frac{6\pi}{\lambda(u_k)}\right],\quad k\in \ZN,
\emd
and
\md0
B_k(x)=\prod_{j=1}^kb_j(x),\quad B(x)=\prod_{j=1}^\infty b_j(x).
\emd
The convergence of the infinite product follows from the bound \e{2.p3}, which will be obtained bellow. Observe that in the process of selection of the numbers \e{2.p9} we were free to define $\delta_k>0$ as small as needed. Besides, taking $u_k$ to be close to $1$ we may get $n_k$ as big as needed. Using these notations and \lem{2.L1}, aside of the conditions \e{2.p9}$-$\e{2.p10} we can additionally claim the bounds
\md3
&\omega\left(2\pi/n_k,B_{k-1}\right)=\sup_{|x-x'|<2\pi/n_k}|B_{k-1}(x)-B_{k-1}(x')|<2^{-k},\label{2.p11}\\
&\left|b_k(x)+1\right|< 2^{-k}, \quad x\in U(n_k,6\delta_k),\label{2.p2}\\
&\left|b_k(x)-1\right|<2^{-k},x\in\ZT\setminus U(n_k,\sqrt[4]{\delta_k}). \label{2.p3}
\emd
From \e{2.p3} we get
\md9\label{2.p7}
\left|B(x)-B_k(x)\right|&=\left|\prod_{j\ge k+1}b_j(x)-1\right|\\
&\le \prod_{j\ge k+1}(1+2^{-j})-1<2^{-k+1},\quad x\in \ZT\setminus \bigcup_{j\ge k+1}U\left(n_j,\sqrt[4]{\delta_j}\right).
\emd
 Take an arbitrary $x\in \ZT$. There exists an integer  $1\le j_0\le n_k$ such that
\md0
\frac{2\pi j_0}{n_k}-x\in \left[\frac{2\pi }{n_k},\frac{4\pi }{n_k}\right]\subset  \left[\frac{2\pi }{n_k},\frac{5\pi }{n_k}\right]\subset\left[\frac{\lambda(u_k)}{3},\lambda(u_k)\right]\subset [\lambda(v_k),\lambda(u_k)],
\emd
where the inclusions follow from the definition of $n_k$ (see \e{2.p22}) and from the inequality $3\lambda(v_k)\le \lambda(u_k)<\pi$ coming from \e{2.p9}. Thus, since $\lambda(r)$ is continuous, we may find numbers $u_k\le r'\le r''\le v_k$, such that
\md1\label{2.p30}
\lambda(r')=\frac{2\pi j_0}{n_k}-x,\quad \lambda(r'')=\frac{2\pi j_0}{n_k}+\frac{\pi }{n_k}-x.
\emd
For the set
\md0
e=\bigcup_{j\ge k+1}U\left(n_j,\sqrt[4]{\delta_j}\right),
\emd
we have
\md0
|e|=10\pi \sum_{j\ge k+1}\sqrt[4]{\delta_j}.
\emd
So taking $r\in [u_k,v_k]$, from \e{2.p10} and  \e{2.p7} we conclude
\md9\label{2.p20}
\big|\Phi_r(x,B)&-\Phi_r(x,B_k)\big|\\
&\le \int_e\varphi_r(x-t)|B(t)-B_k(t)|dt+2^{-k+1} \int_{ \ZT\setminus e}\varphi_r(x-t)dt\\
&\le 2\cdot 2^{-k}+2^{-k+1}=4\cdot 2^{-k} ,\quad x\in \ZT.
\emd
If
\md0
t\in I=(-\delta_k\lambda(u_k),\delta_k\lambda(u_k))\subset \left(-\frac{6\pi\delta_k}{n_k},\frac{6\pi\delta_k}{n_k}\right),
\emd
then we have
\md2
&\frac{2\pi j_0}{n_k}-t\in U(n_k,6\delta_k),\\
 &\frac{2\pi j_0}{n_k}+\frac{\pi}{n_k}-t\in\ZT\setminus U(n_k,\sqrt[4]{\delta_k}).
\emd
Then, using these relations, \e{2.p2} and \e{2.p11}, we get
\md9\label{2.p12}
\left|B_k\left(\frac{2\pi j_0}{n_k}-t\right)\right.&+\left. B_{k-1}\left(\frac{2\pi j_0}{n_k}\right)\right|\\
&\le\left|B_{k-1}\left(\frac{2\pi j_0}{n_k}-t\right)\right|\left|b_k\left(\frac{2\pi j_0}{n_k}-t\right)+1\right|\\
&+\left|B_{k-1}\left(\frac{2\pi j_0}{n_k}-t\right)-B_{k-1}\left(\frac{2\pi j_0}{n_k}\right)\right|\\
&< 2^{-k}+2^{-k}=2^{-k+1}
\emd
and
\md9\label{2.p13}
\left|B_k\left(\frac{2\pi j_0}{n_k}+\frac{\pi}{n_k}-t\right)\right.&-\left. B_{k-1}\left(\frac{2\pi j_0}{n_k}\right)\right|\\
&\le\left|B_{k-1}\left(\frac{2\pi j_0}{n_k}+\frac{\pi}{n_k}-t\right)\right|\left|b_k\left(\frac{2\pi j_0}{n_k}+\frac{\pi}{n_k}-t\right)-1\right|\\
&+\left|B_{k-1}\left(\frac{2\pi j_0}{n_k}+\frac{\pi}{n_k}-t\right)-B_{k-1}\left(\frac{2\pi j_0}{n_k}\right)\right|\\
&< 2^{-k}+2^{-k}=2^{-k+1}.
\emd
On the other hand, using \e{2.p8}, \e{2.p30} and \e{2.p12}, we get
\md9\label{2.c15}
&\left|\Phi_{r'}(x+\lambda(r'),B_k) + B_{k-1}\left(\frac{2\pi j_0}{n_k}\right)\right|\\
&\mspace{150mu}= \left|\int_\ZT\varphi_{r'}(t)B_k(x+\lambda(r')-t)dt+B_{k-1}\left(\frac{2\pi j_0}{n_k}\right)\right|\\
&\mspace{150mu}=  \left|\int_\ZT\varphi_{r'}(t)\left[B_k\left(\frac{2\pi j_0}{n_k}-t\right)+B_{k-1}\left(\frac{2\pi j_0}{n_k}\right)\right]dt\right|\\
&\mspace{150mu}\le  \left|\int_I\varphi_{r'}(t)\left[B_k\left(\frac{2\pi j_0}{n_k}-t\right)+B_{k-1}\left(\frac{2\pi j_0}{n_k}\right)\right]dt\right|\\
&\mspace{150mu}+\left|\int_{I^c}\varphi_{r'}(t)\left[B_k\left(\frac{2\pi j_0}{n_k}-t\right)+B_{k-1}\left(\frac{2\pi j_0}{n_k}\right)\right]dt\right|\\
&\mspace{150mu}\le 2^{-k+1}\int_I\varphi_{r'}(t)dt+2\cdot 2^{-k}\le 4\cdot 2^{-k}.
\emd
Similarly, using \e{2.p13}, we conclude
\md1\label{2.c16}
\left|\Phi_{r''}\left(x+\lambda(r''),B_k\right)-B_{k-1}\left(\frac{2\pi j_0}{n_k}\right)\right|\le 4\cdot 2^{-k}.
\emd
From \e{2.p20}, \e{2.c15} and \e{2.c16} it follows that
\md0
\left|\Phi_{r'}(x+\lambda(r'),B)-\Phi_{r''}\left(x+\lambda(r''),B\right)\right|\ge 1-16\cdot 2^{-k},
\emd
which implies the divergence of $ \Phi_{r}(x+\lambda(r),B)$ at a point $x$. The theorem is proved.
\end{proof}

%% file: sections/3.1.tex
This chapter is devoted to differentiation bases in $\ZR^n$, which is defined as follows:

\begin{definition}\label{3.def-diff-bases}
A family $\ZM$ of bounded, positively measured sets from $\ZR^n$ is said to be a differentiation basis (or simply basis), if for any point $x\in \ZR^n$ there exists a sequence of sets $E_k\in\ZM$ such that $x\in E_k$,  $k=1,2,\ldots$ and $\diam(E_k)\to 0$ as $k\to\infty$.
\end{definition}

We remind the classical theorems determining the optimal Orlicz space for the functions having almost everywhere differentiable integrals with respect to the bases of rectangles $\zR^n$.
\begin{customtheorem}{\ref{OldJMZ}}[Jessen--Marcinkiewicz--Zygmund, \cite{JMZ}]
$L(1+\log^+ L)^{n-1}(\ZR^n)\subset \ZF(\zR^n).$
\end{customtheorem}

\begin{customtheorem}{\ref{OldSaks}}[Saks, \cite{Saks}] If the convex function $\Psi:\ZR^+\to\ZR^+$ satisfies
\md0
\Psi(t)=o(t\log^{n-1} t)\text { as } t\to\infty,
\emd
then $\Psi(L)(\ZR^n)\not\subset \ZF(\zR^n)$. Moreover, there exists a positive function $f\in \Psi(L)(\ZR^n)$ such that $\delta_{\zR^n}(x,f)=\infty$ everywhere.	
\end{customtheorem}

The optimal Orlicz space remains the same if we consider the basis $\DR^n$ instead of $\zR^n$. The first part follows from the embedding $L(1+\log^+ L)^{n-1}(\ZR^n)\subset \ZF(\zR^n)\subset \ZF(\DR^n)$ and the second can be deduced from the following
\begin{customtheorem}{\ref{zerekidze}}[Zerekidze, \cite{Zer1} (see also \cite{Zer2, Zer3})]
$\ZF^+(\DR^n)= \ZF^+(\zR^n).$
\end{customtheorem}

However, the set of functions having almost everywhere differentiable integrals with respect to these bases can differ. In \sect{dyadic-R2} we prove that the condition $\gamma_\Delta<\infty$ is necessary and sufficient for the full equivalency of rare dyadic basis $\DR^2_\Delta$ and complete dyadic bases $\DR^n$.

\begin{theorem}\label{TX2}
If $\Delta=\{\nu_k\}$ is an increasing sequence of positive integers with
\md0
\gamma_\Delta = \sup_{k\in \ZN}(\nu_{k+1}-\nu_k)<\infty,
\emd
then
\md0
\ZF(\DR^2_\Delta)= \ZF(\DR^2).
\emd
\end{theorem}

\begin{theorem}\label{TX2.1}
If $\Delta=\{\nu_k\}$ is an increasing sequence of positive integers with $\gamma_\Delta=\infty$, then there exists a function $f\in\ZF(\DR^2_\Delta)$ such that
\md0
\limsup_{\length(R)\to 0,\, x\in R\in\DR^2}\left|\frac{1}{|R|}\int_R f(t)\,dt\right| = \infty
\emd
for any $x\in\ZR^n$.
\end{theorem}

In \sect{quasi-equivalent} we prove that two quasi-equivalent subbases of some density basis differentiate the same class of non-negative functions. In \sect{applications} we apply this theorem for bases formed of rectangles.

\begin{definition}\label{3.def-density-bases}
A basis $\ZM$ is said to be density basis if $\ZM$ differentiates the integral of any characteristic function $\ZI_E$ of measurable set $E$:
\md0
\delta_\ZM(x,\ZI_E)=0 \text{ at almost every } x\in\ZR^n.
\emd
We will say that the basis $\ZM$ differentiates a class of functions $\ZF$, if basis $\ZM$ differentiates the integrals of all functions of $\ZF$.
\end{definition}

\begin{definition}\label{3.def-quasi-equiv-bases}
Let $\ZM_1,\ZM_2\subseteq\ZM$ be subbases. We will say that basis $\ZM_2$ is quasi-coverable by basis $\ZM_1$ (with respect to basis $\ZM$) if for any $R\in\ZM_2$ there exist $R_k\in\ZM_1,\,k=1, 2, \ldots, p$ and $R'\in\ZM$ such that
\md3
& R\subseteq \tilde{R} \subseteq R',\quad \tilde{R}=\bigcup_{k=1}^p R_k\label{qc1}\\
& \diam(R')\le c\cdot\diam(R),\quad |R'|\le c|R_k|,\quad k=1, 2, \ldots, p,\label{qc2}\\
& \sum_{k=1}^p|R_k|\le c|\tilde{R}|,\quad |\tilde{R}|\le c|R|,\label{qc0}
\emd
where constant $c\ge1$ depends only on bases $\ZM_1, \ZM_2$ and $\ZM$. We will say two bases are quasi-equivalent if they are quasi-coverable with respect to each other.
\end{definition}

\begin{theorem}\label{quasi-equiv}
Let $\ZM_1$ and $\ZM_2$ be subbases of density basis $\ZM$ formed of open sets from $\ZR^n$. If the bases $\ZM_1$ and $\ZM_2$ are quasi-equivalent with respect to $\ZM$ then
\md0
\ZF^+(\ZM_1)= \ZF^+(\ZM_2).
\emd
\end{theorem}

%% file: sections/3.2.tex
Denote by  $\overline{E}$ and $\mathring E$ the closure and the interior of a set $E\subset \ZR^2$ respectively, $\ZI_E$ denotes the indicator function of $E$. For a given rectangle  $R\in\zR^2$ we denote by  $\length(R)$ the length of the bigger side of $R$. A set $E\subset \ZR^2$ is said to be simple, if it can be written as a union of squares of the form
\md0
\left[\frac{i-1}{2^n},\frac{i}{2^n}\right)\times \left[\frac{j-1}{2^n},\frac{j}{2^n}\right),\quad i,j,n\in\ZZ.
\emd
If $n$ is the minimal integer with this relation, then we write $\width(E)=2^{-n}$. Note that if $E$ is a dyadic rectangle, then $\width(E)$ coincides with the length of the smaller side of $E$. If $E$ is a square, then $\length(E)=\width(E)$.
Denote
\md3
E_{ij}(n)
&= \bigcup_{k=0}^{n-1}\left[\frac{i}{2},\frac{i}{2}+\frac{1}{2^{k+1}}\right)\times\left[\frac{j}{2},\frac{j}{2}+\frac{1}{2^{n-k}}\right),\label{a-4}\\
F_{ij}(n)
&= \left[\frac{i}{2},\frac{i}{2}+\frac{1}{2^{n}}\right)\times \left[\frac{j}{2},\frac{j}{2}+\frac{1}{2^{n}}\right)\nonumber\\
&=\bigcap_{k=0}^{n-1}\left[\frac{i}{2},\frac{i}{2}+\frac{1}{2^{k+1}}\right)\times\left[\frac{j}{2},\frac{j}{2}+\frac{1}{2^{n-k}}\right)\subset E_{ij}(n),\quad i,j=0,1,\nonumber
\emd
and define the sets
\md3
E(n) &= E_{00}(n)\cup E_{01}(n) \cup E_{10}(n) \cup E_{11}(n),\label{a-5}\\
F(n) &= F_{00}(n)\cup F_{01}(n) \cup F_{10}(n) \cup F_{11}(n)\subset E(n).
\emd
Introduce the functions
\md2
u(x,n) &= (n+1)2^{n-2}\left(\ZI_{F_{00}(n)}(x)+\ZI_{F_{11}(n)}(x)-\ZI_{F_{10}(n)}(x)-\ZI_{F_{01}(n)}(x)\right),\quad n\in\ZN,\\
v(x) &= \ZI_{(0,1/2)\times(0,1/2)}(x)+\ZI_{(1/2,1)\times(1/2,1)}(x)-\ZI_{(0,1/2)\times(1/2,1)}(x)-\ZI_{(1/2,1)\times(0,1/2)}(x).
\emd
Let $\omega\in \ZQ^2$ be an arbitrary square and $\phi_\omega$ be the linear transformation of $\ZR^2$ taking $\omega$ onto unit square $[0,1)^2\subset \ZR^2$. For an arbitrary function $f(x)$ defined on $[0,1)^2$ and for a set $E\subset [0,1)^2$ we define
\md0
f_\omega(x)=f(\phi_\omega(x)),\quad E_\omega=(\phi_\omega)^{-1}(E)\subset \omega.
\emd
We have
\md3
&\supp(u_\omega(x,n))=F_\omega(n),\label{x1}\\
&\supp(v_\omega(x))=\omega,\label{x2}\\
&|E_\omega(n)|=\frac{(n+1)|\omega|}{2^n},\quad |F_\omega(n)|=\frac{|\omega|}{4^{n-1}},\label{u3}\\
&\width\left(E_\omega(n)\right)=\width\left(F_\omega(n)\right)=\width(\omega)\cdot 2^{-n}.\label{x5}
\emd
Simple calculations show that
\md3
&\|u_\omega(x,n)\|_1=|E_\omega(n)|=\frac{n+1}{2^n}|\omega|,\label{a-6}\\
&\|v_\omega(x)\|_1=|\omega|.\label{aa-6}
\emd
Then observe that, if $\omega\in\DQ^2$ is a dyadic square, then for any point $x\in E_\omega(n)$ there exists a dyadic rectangle $R(x)\in \DR^2$ with
\md3
&\frac{1}{|R(x)|}\left|\int_{R(x)} u_\omega(x,n)dx\right|=\frac{n+1}{2},\quad x\in R(x)\subset E_\omega(n),\label{a-7}\\
&\width(R(x))=\width(\omega)\cdot 2^{-n}.
\emd
Moreover, the rectangle $R(x)$ coincides with $(\phi_\omega)^{-1}$-image of one of the representation rectangles from \e{a-4}. Similarly, if $\omega\in\DQ^2$, then
\md3
&\frac{1}{|R(x)|}\left|\int_{R(x)} v_\omega(x)dx\right|=1,\quad x\in R(x)\subset \omega,\label{a-8}\\
&\width(R(x))=\frac{\width(\omega)}{2}.
\emd
for some square $R(x)$ with $|R(x)|=|\omega|/4$. In this case $R(x)$ coincides with one of the four squares forming $\omega$.

The following simple lemma has been proved in \cite{KarKar, KarKarSaf}.
\begin{lemma}\label{L-3}
Let $Q\in \DQ^2$ be an arbitrary dyadic square, $f(x)=f(x_1,x_2)\in L^1(\ZR^2)$ be a function with $\supp f(x)\subset Q$ and
\md3
\int_\ZR f(x_1,t)\,dt=\int_\ZR f(t,x_2)\,dt=0, \quad x_1,x_2\in \ZR.\label{b-1}
\emd
Then for any dyadic rectangle $R\in\DR^2$ satisfying $\mathring{R}\not\subset Q$ we have
\md1\label{1-21}
\int_R f(x)\,dx=0 .
\emd
\end{lemma}
\begin{proof}
We suppose
\md0
Q=[\alpha_1,\beta_1)\times[\alpha_2,\beta_2),\quad R=[a_1,b_1)\times[a_2,b_2).
\emd
If $R\cap Q=\varnothing$, then \e{1-21} is trivial. Otherwise we will have either $[\alpha_1,\beta_1)\subset[a_1,b_1)$ or $[\alpha_2,\beta_2)\subset[a_2,b_2)$. In the first case, using \e{b-1}, we get
\md8
\int_R f(x)\,dx
&=\int_{a_2}^{b_2}\int _{a_1}^{b_1}f(x_1,x_2)\,dx_1\,dx_2\\
&=\int_{a_2}^{b_2}\int _{\alpha_1}^{\beta_1}f(x_1,x_2)\,dx_1\,dx_2\\
&=\int_{a_2}^{b_2}\left(\int _\ZR f(x_1,x_2)\,dx_1\right)\,dx_2=0.
\emd
The second case is proved similarly.
\end{proof}
\begin{lemma}\label{L0}
Let $m$ be a positive integer and $Q$ be a dyadic square. Then for any simple set $E\varsubsetneq[0,1)^2$, there exists a finite family $\Omega$ of dyadic squares $\omega\subset Q$ such that
\md5
E_\omega\cap E_{\omega'}=\varnothing,\quad \omega\neq\omega',\label{u0}\\
\min_{\omega\in\Omega}\width(\omega)= \width(Q)\cdot (\width(E))^{m},\label{u1}\\
\left|Q\setminus \bigcup_{\omega\in \Omega}E_{\omega}\right|= |Q|\left(1-|E|\right)^m.\label{u2}
\emd
\end{lemma}
\begin{proof}
Define a sequence of sets $G_k$, $k=1,2,\ldots,m$, with
\md1\label{c-8}
Q=G_1\supset G_2\supset \ldots\supset G_m,
\emd
and finite families of dyadic squares $\Omega_k\subset \DQ^2$, $k=1,2,\ldots, m+1$, such that
\md3
\width(\omega) &= \width(Q)\cdot  (\width(E))^{k-1},\quad \omega \in \Omega_k,\quad k=1,2,\ldots, m+1,\label{d43}\\
G_k &= \bigcup_{\omega\in\Omega_k}\omega,\quad k=1,2,\ldots, m+1,\label{c-6}\\
G_k &= G_{k-1}\setminus \bigcup_{\omega\in \Omega_{k-1}}E_\omega=\bigcup_{\omega\in \Omega_{k-1}}\left(\omega\setminus E_\omega\right),\quad k=2,\ldots, m+1.\label{c-7}
\emd
We do it by induction. For the first step of induction we take just $G_1=Q$ and let $\Omega_1$ consist of a single rectangle $Q$. Suppose we have already chosen the sets $G_k$ and the families $\Omega_k$  for $k=1,2,\ldots,p$, satisfying \e{c-8}$-$\e{c-7}. Set
\md0
G_{p+1}=G_p\setminus  \bigcup_{\omega\in \Omega_{p}}E_\omega=\bigcup_{\omega\in \Omega_{p}}\left(\omega\setminus E_\omega\right).
\emd
From the induction hypothesis of \e{d43} it follows that
\md0
\width\left(\omega\setminus E_\omega\right)=\width(\omega)\cdot \width(E)= \width(Q)\cdot  (\width(E))^p.
\emd
Hence we conclude that $G_{p+1}$ is a union of dyadic squares with side lengths $\width(Q)\cdot  (\width(E))^p$ and we define the family $\Omega_{p+1}$ as a collection of these squares. Thus we get $G_{p+1}$ and $\Omega_{p+1}$ satisfying the conditions \e{c-8}$-$\e{c-7} for $k=p+1$, that completes the induction process.
Applying \e{a-6}, \e{c-6} and \e{c-7} we obtain
\md0
|G_{k}|=|G_{k-1}|-\left|\bigcup_{\omega\in \Omega_{k-1}  }E_\omega\right|=|G_{k-1}|-|E||G_{k-1}|=\left(1-|E|\right)|G_{k-1}|
\emd
and therefore
\md1\label{1-11}
\left|G_{m+1}\right|=\left(1-|E|\right)^m|Q|.
\emd
Obviously the family of squares $\Omega=\cup_{k=1}^{m+1}\Omega_k$ satisfies the hypothesis of the lemma. Indeed, suppose $\omega,\omega'\in \Omega$ are arbitrary squares. If $\omega,\omega'\in\Omega_k$ for some $k$, then according to \e{d43} we have $\omega\cap\omega'=\varnothing$ and so \e{u0}. If $\omega\in \Omega_k$, $\omega'\in \Omega_{k'}$ and $k< k'$, then
\md2
&E_{\omega'}\subset \omega'\subset G_{k'},\\
&E_{\omega}\subset G_{k}\setminus G_{k+1} \Rightarrow E_\omega\cap G_{k'}=\varnothing.
\emd
Thus we again get \e{u0}. The condition \e{u1} immediately follows from \e{d43},  and \e{u2} follows from \e{1-11} and from the relation
\md0
\left|\bigcup_{\omega\in \Omega}E_{\omega}\right|=\left|\bigcup_{k=1}^{m+1}\bigcup_{\omega\in\Omega_k}E_{\omega}\right|=\left|\bigcup_{k=1}^{m+1}G_{k}\setminus G_{k+1}\right|=\left|Q\setminus G_{m+1}\right|=|Q|(1-\left(1-|E|\right)^m).
\emd
\end{proof}
\begin{lemma}\label{L-4}
Let $L>1$ be a positive integer and $Q\in\DQ^2$ be a dyadic square.
Then there exist a function $f\in L^\infty(\ZR^2)$, numbers $\alpha(L)\in \ZN$ and $\beta(L)>0$, depended on $L$, such that
\md3
&\supp f\subset Q, \label{c-1}\\
&\|f\|_\infty\le \beta(L),\label{c-2}\\
&|\supp f|\le\frac{ 2|Q|}{\beta(L)},\label{x10}\\
&\width(\supp f)\ge \width(Q)\cdot 2^{-\alpha(L)}, \label{x11}\\
&\int_R f(x)dx=0,\quad R\in \DR^2,\quad \mathring{R}\not\subset Q,\label{c-4}
\emd
and for any point $x\in Q$ there exists a rectangle $R(x)\subset Q$ satisfying
\md3
&\width(R(x))\ge \width(Q)\cdot 2^{-\alpha(L)} ,\label{x15}\\
&\frac{1}{|R(x)|}\left|\int_{R(x)} f(t)dt\right|\ge L.\label{c-5}
\emd
\end{lemma}
\begin{proof}
Let $n=2L$ and denote
\md3
&\alpha(L)=n(2^n+1),\quad \beta(L)=(n+1)2^{n-2},\label{u4}\\
&m=m(L)=\left[\frac{2^n(\ln (n+1)+(n-2)\ln 2)}{n+1}\right]+1<2^n.\label{u5}
\emd
Let $E=E(n)$ be the set defined in \e{a-5}. We have $|E(n)|=(n+1)/2^n$ and $\width(E(n))=2^{-n}$. Applying \lem{L0}, we may find family $\Omega$ of dyadic squares $\omega\subset Q$ with properties \e{u0}$-$\e{u2}.
Set
\md1\label{u6}
G=\bigcup_{\omega\in\Omega}E_\omega(n),\quad G_1=Q\setminus G.
\emd
According to \e{u2}, \e{u4} and \e{u5}, we have
\md0
|G_1|=\left(1-|E(n)|\right)^m|Q|=\left(1-\frac{n+1}{2^n}\right)^m|Q|<\frac{|Q|}{\beta(L)}
\emd
From \e{u1} and \e{u5} it follows that
\md0
G_1=\bigcup_{\omega\in \Omega_1}\omega,
\emd
where $\Omega_1$ is a family of squares with
\md1\label{u7}
\min_{\omega\in\Omega_1}\width(\omega) =\min_{\omega\in\Omega}\width(\omega)=\width(Q)\cdot (\width(E(n)))^{m}
\ge \width(Q)\cdot 2^{-n\cdot 2^n}.
\emd
Define
\md0
f(x)=\sum_{\omega\in \Omega}u_{\omega}(x,n)+\beta(L)\sum_{\omega\in \Omega_1}v_{\omega}(x)=g(x)+g_1(x).
\emd
Clearly this function satisfies \e{c-1} and \e{c-2}. Then, we have
\md2
\supp g &= \bigcup_{\omega\in \Omega}F_\omega(n)\subset G, \quad \supp g_1=G_1,\\
\supp f &= \supp g\bigcup \supp g_1.
\emd
This together with \e{u3} and \e{u6}  implies
\md2
|\supp f|
&=\bigcup_{\omega\in \Omega}|F_\omega(n)|+|G_1|\\
&=\frac{1}{(n+1)2^{n-2}}\sum_{\omega\in \Omega}|E_\omega(n)|+|G_1|\\
&=\frac{1}{(n+1)2^{n-2}}|G|+|G_1|\le \frac{2|Q|}{\beta(L)}
\emd
and therefore we get \e{x10}. Using \e{u7}, we obtain
\md2
\width(\supp g) &\ge \min_{\omega\in\Omega}\width(\omega)\cdot \width(F(n))= \width(Q)\cdot 2^{-n(2^n+1)}= \width(Q)\cdot 2^{-\alpha(L)},\\
\width(\supp g_1) &\ge \min_{\omega\in\Omega_1}\width(\omega)\ge \width(Q)\cdot 2^{-n\cdot 2^n}> \width(Q)\cdot 2^{-\alpha(L)},
\emd
and therefore we get \e{x11}. The condition \e{c-4} follows from \lem{L-3}, since $f(x)$ satisfies the condition \e{b-1} according the definitions of functions $u_\omega(x,n)$ and $v_\omega(x)$. To prove \e{c-5} we take an arbitrary point $x\in Q$. We have either $x\in G$ or $x\in G_1$. In the first case we will have $x\in E_{\omega}(n)$ for some square $\omega\in\Omega$. By \e{a-7} there exists a dyadic rectangle $R=R(x)$, $x\in R\subset E_{\omega}(n)$, such that
\md0
\frac{1}{|R|}\left|\int_Rf(t)dt\right|=\frac{1}{|R|}\left|\int_R u_{\omega}(t,n)dt\right|=\frac{n+1}{2}>L.
\emd
In the second case from \e{a-8} we obtain
\md0
\frac{1}{|R|}\left|\int_{R}f(t)dt\right|=\frac{\beta(L)}{|R|}\left|\int_R v_{\omega}(t)dt\right|\ge 2^n>L
\emd
for some square $R=R(x)$, $x\in R\subset \omega$. Obviously in any case $R(x)$ satisfies \e{x15}.  Lemma is proved.
\end{proof}

%% file: sections/3.3.tex
\begin{proof}[Proof of \trm{TX2}]
Let $\Delta=\{\nu_k\}$ be a sequence with $\gamma_\Delta<\infty$. Suppose conversely, we have
\md0
\ZF(\DR^2_\Delta)\setminus \ZF(\DR^2)\neq\varnothing.
\emd
That means there exist a function $f\in L_{\rm loc}(\ZR^2)$, a number $\alpha>0$ and a set $E\subset\ZR^n$ with $|E|>0$ such that
\md3
\delta_{\DR^2_\Delta}(x, f) &= 0,\quad\text{ a.e.},\label{xd5}\\
\delta_{\DR^2}(x, f) &> \alpha,\quad x\in E,\label{d4}
\emd
According to \e{xd5} for almost any $x\in \ZR^2$ one can choose a number $\delta(x)>0$ such that
the conditions
\md0
x\in R\in \DR^2_\Delta, \quad \length(R)<\delta(x),
\emd
imply
\begin{equation}\label{cond}
\left|\frac{1}{|R|}\int_R f -f(x) \right|<\frac{\alpha}{2}.
\end{equation}
For some $\delta>0$ the set $F=\{x\in E\colon
\delta(x)\ge\delta\}\subset E$ has positive measure.  Then, using the representation
\begin{equation*}
F=\bigcup_{j\in\ZZ }\left\{x\in F \colon \frac{j\alpha}{2}\le f(x) <
\frac{(j+1)\alpha}{2} \right\},
\end{equation*}
we find a set
\md1\label{x112}
G=\left\{x\in F \colon  \frac{j_0\alpha}{2}\le f(x) <\frac{(j_0+1)\alpha}{2}\right\}\subset F
\emd
having positive measure. Combining \e{d4}, \e{cond} and \e{x112}, we will have
\md3
&\delta_{\DR^2}(x, f) >\alpha,\quad x\in  G\label{E-1},\\
&\left|\frac{1}{|R|}\int_R f -f(x) \right|<\frac{\alpha}{2},\text{ if } x\in R\cap G,\, R\in \DR^2(\Delta),\,
\length(R)<\delta,\label{E-2}\\
&\sup_{x,y\in G}|f(x)-f(y)|\le\frac{\alpha}{2}.\label{E-3}
\emd
Since almost all points of  $G$ are density points, we may fix $x_0\in G$ with
\begin{equation*}
\lim_{\length(R)\to 0,\, x_0\in R\in\DR^2} \frac{|R\cap
G|}{|R|} = 1.
\end{equation*}
Using this relation and  (\ref{E-1}), we find a rectangle
\md0
R'=\left[\frac{p-1}{2^{n}},\frac{p}{2^{n}}\right)\times\left[\frac{q-1}{2^{m}},\frac{q}{2^{m}}\right),
\emd
such that
\md3
&x_0\in R'\in \DR^2,\quad \length(R')<\delta,\label{d10}\\
&\left|\frac{1}{|R'|}\int_{R'} f - f(x_0)\right|>\alpha,\label{xd11}\\
&|R'\cap G| > (1-4^{-\gamma_\Delta})|R'|,\label{xd12}
\emd
Besides, we may suppose
\md1\label{x113}
\nu_{k_t-1}<n\le \nu_{k_t},\quad \nu_{k_s-1}<m\le \nu_{k_s},
\emd
for some integers $t$ and $s$. This and the definition of $\gamma_\Delta$ imply that $R'$ is a union of rectangles of the form
\md0
\left[\frac{i-1}{2^{\nu_{k_{t}}}},\frac{i}{2^{\nu_{k_{t}}}}\right)\times\left[\frac{j-1}{2^{\nu_{k_s}}},\frac{j}{2^{\nu_{k_s}}}\right)\in  \DR^2_\Delta,
\emd
and from (\ref{xd11}) it follows that at least for one of these rectangles, say $R''$, we have
\md1\label{x114}
\left|\frac{1}{|R''|}\int_{R''} f - f(x_0)\right|>\alpha.
\emd
From the definition of $\gamma_\Delta$ and \e{x113} we get
\md0
|R''|=\frac{1}{2^{ \nu_{k_t}+ \nu_{k_s}}}\ge \frac{1}{2^{\nu_{k_{t}}+\nu_{k_{s}}-\nu_{k_t-1}-\nu_{k_s-1}}}\cdot \frac{1}{2^{n+m}}\ge |R'|\cdot 4^{-\gamma_\Delta}.
\emd
From this and (\ref{xd12}) we obtain $R''\cap G\neq\varnothing$. Take a point
$x_1\in R''\cap G$. From (\ref{E-3}) and \e{x114} we get
\md1\label{xd20}
\left|\frac{1}{|R''|}\int_{R''} f - f(x_1)\right| >\left|\frac{1}{|R''|}\int_{R''} f - f(x_0)\right| -|f(x_1)-f(x_0)|>\frac{\alpha}{2}.
\emd
On the other hand we have $x_1\in R''\cap G$, $R''\in \DR^2_\Delta$,
$\length(R'')\le\length(R')<\delta_0$, and therefore by (\ref{E-2}) we obtain
\md0
\left|\frac{1}{|R''|}\int_{R''} f - f(x_1) \right|<\alpha/2.
\emd
The last relation together with \e{xd20} gives a contradiction, which completes the proof of the theorem.
\end{proof}

\begin{proof}[Proof of \trm{TX2.1}]
Now we suppose $\gamma_\Delta=\infty$, which means there exists a sequence of integers $p_k\nearrow\infty$ such that
\md1\label{d16}
\lim_{k\to\infty}(\nu_{p_k+1}-\nu_{p_k})=\infty.
\emd
Using this relation, we may find sequences of integers $L_k$ and $l_k$, $k=1,2,\ldots $, such that
\md3
&l_{k+1}> l_k+\alpha(L_k),\quad k=1,2,\ldots,\label{d25}\\
&\nu_{p_k}<l_k<l_k+\alpha(L_k)<\nu_{p_k+1},\quad k=1,2,\ldots,\label{d26}\\
&L_{k+1}>2^k\cdot (\beta(L_k)+k)\quad k=1,2,\ldots,\label{d28}
\emd
where $\alpha(L)$ and $\beta(L)$ are the constants taken from \lem{L-4}. Applying \lem{L-4} for the numbers $L=L_k$, $l=l_k$ and for the square
\md0
Q=Q_{ij}^k=\left[\frac{i-1}{2^{l_k}},\frac{i}{2^{l_k}}\right)\times\left[\frac{j-1}{2^{l_k}},\frac{j}{2^{l_k}}\right),\quad 1\le i,j\le 2^{l_k},
\emd
we get functions $f_{ij}^k\in L^\infty(\ZR^2)$ satisfying the conditions
\md3
&\supp f_{ij}^k\subset Q_{ij}^k, \label{xc-1}\\
&\|f_{ij}^k\|_\infty\le \beta(L_k),\label{xc-2}\\
&|\supp f_{ij}^k|\le \frac{2|Q_{ij}^k|}{\beta(L_k)},\label{xx10}\\
&\width(\supp f_{ij}^k)\ge 2^{-l_k-\alpha(L_k)}, \label{xx11}\\
&\int_R f_{ij}^k(x)dx=0,\quad R\in \DR^2,\quad \mathring{R}\not\subset Q_{ij}^k,\label{xc-4}
\emd
and for any point $x\in Q_{ij}^k$ there exists a dyadic rectangle $R_k(x)\subset Q_{ij}^k$ with
\md3
&\width(R_k(x))\ge 2^{-l_k-\alpha(L_k)},\label{xx15-}\\
&\frac{1}{|R_k(x)|}\left|\int_{R_k(x)} f_{ij}^k(t)dt\right|\ge L_k.\label{xxxc-5}
\emd
Define the function
\md0
F_k(x)=\sum_{i,j=1}^{2^{l_k}}f_{ij}^k(x).
\emd
From the relations \e{xc-1}$-$\e{xxxc-5} we conclude
\md3
&|\supp F_k|\le  \frac{2}{\beta(L_k)},\label{d19}\\
&\width(\supp F_k)\ge 2^{-l_k-\alpha(L_k)}, \label{d20}\\
&\|F_k\|_\infty\le \beta(L_k),\label{d21}\\
&\int_R F_k(x)dx=0,\quad R\in \DR^2,\quad \length(R)\ge 2^{-l_k},\label{d23}
\emd
and for any point $x\in [0,1)^2$ there exists a dyadic rectangle $R_k(x)\subset [0,1)^2$ such that
\md3
&2^{-l_k}>\length(R_k(x))\ge \width(R_k(x))\ge  2^{-l_k-\alpha(L_k)},\label{xx15}\\
&\frac{1}{|R_k(x)|}\left|\int_{R_k(x)} F_k(t)dt\right|\ge L_k.\label{xc-5}
\emd
Denote
\md1\label{d29}
F(x)=\sum_{k=1}^\infty \frac{F_k(x)}{2^k}.
\emd
From \e{d19} and \e{d26} it follows that $\|F_k\|_1\le 2$ and so $\|F\|_1\le 2$. Let $x\in [0,1)^2$  be an arbitrary point. From the relations \e{d25} and \e{xx15} we get
$\length(R_k(x))\ge  2^{- l_{k+1}}\ge 2^{-l_j}$ if $j>k$. Thus, using \e{d23}, we obtain
\md1\label{dd30}
\int_{R_k(x)} F_j(t)dt=0,\quad j>k.
\emd
On the other hand the relations \e{d21} and \e{d28} imply
\md1\label{d31}
\left|\frac{1}{|R_k(x)|}\int_{R_k(x)} \sum_{j=1}^{k-1}\frac{F_j(t)}{2^j}dt\right|\le \beta(L_{k-1})<\frac{L_k}{2},\quad k\ge 2.
\emd
From \e{xc-5}, \e{dd30} and \e{d31} we get the inequality
\md0
\left|\frac{1}{|R_k(x)|}\int_{R_k(x)} F(t)dt\right|\ge \frac{1}{|R_k(x)|}\left|\int_{R_k(x)} F_k(t)dt\right|-\frac{L_k}{2}>\frac{L_k}{2},
\emd
which yields
\md1\label{d40}
\limsup_{\length(R)\to 0,\, x\in R\in\DR^2}\left|\frac{1}{|R|}\int_R F(t)dt\right|=\infty,\quad x\in [0,1)^2.
\emd
Now take an arbitrary rectangle $R\in \DR^2_\Delta$. We have
\md1\label{d32}
\length(R)=2^{-\nu_k}\ge \width(R)=2^{-\nu_t}.
\emd
From \e{d23} we get
\md1\label{d33}
\int_{R} F_j(t)dt=0\,\text { if  } l_j\ge \nu_k.
\emd
On the other hand if $l_j<\nu_k$, then from \e{d26} it follows that
\md0
l_j+\alpha(L_j)<\nu_k
\emd
and therefore by \e{d20} we get
\md1\label{d35}
\width(\supp(F_j))\ge 2^{-l_j-\alpha(L_j)}\ge 2^{-\nu_k} .
\emd
Thus, using  simple properties of dyadic rectangles, we conclude that
\md1\label{d34}
l_j<\nu_k,\, R\not\subset \supp(F_j)  \Rightarrow  R\cap \supp(F_j)=\varnothing.
\emd
Consider the sets
\md2
&G_1=\{x\in [0,1)^2:\, \delta_\zR(x,F_k)=0,\,k=1,2,\ldots \},\\
&G_2=\bigcup_{k=1}^\infty\bigcap_{j:\,l_j\ge \nu_k}^\infty \bigg([0,1)^2\setminus \supp(F_j)\bigg),\\
&G=G_1\cap G_2.
\emd
Since $F_k(x)$ is bounded, the equality $\delta_\zR(x,F_k)=0$ holds almost everywhere and so  $|G_1|=1$. From \e{d19} it follows that $|G_2|=1$ and therefore we get $|G|=1$. Take an arbitrary point $x\in G$. We have
\md1\label{d36}
x\not\in \supp(F_j),\quad j>k_0,
\emd
for some $k_0$. Consider the rectangle $R\in \DR^2_\Delta$ such that $x\in R$. Suppose we have \e{d32} and $k>k_0$. Then form \e{d34} and \e{d36} we get
\md1\label{d37}
R\cap \supp(F_j)=\varnothing,\text { if } j>k_0\text { and } l_j<\nu_k.
\emd
From \e{d33} and \e{d37} we conclude
\md0
\frac{1}{|R|}\int_{R} F(t)dt=\sum_{j=1}^{k_0} \frac{1}{2^j\cdot |R|}\int_{R} F_j(t)dt.
\emd
Thus we obtain
\md1\label{d41}
\lim_{\length(R)\to 0,\, x\in R\in\DR^2_\Delta}\frac{1}{|R|}\int_{R} F(t)dt=\sum_{j=1}^{k_0}  \frac{F_j(x)}{2^j}.
\emd
On the other hand \e{d36} implies
\md1\label{d42}
F(x)=\sum_{j=1}^{k_0}  \frac{F_j(x)}{2^j}.
\emd
From \e{d41} and \e{d42} we conclude that $F\in \ZF(\DR^2_\Delta)$ and $\supp F\subset[0,1)^2$. To have a function $f$ defined on entire $\ZR^2$ we set
\md0
f(x) = f(x_1,x_2) = F\left(\{x_1\},\{x_2\}\right),\quad x\in\ZR^2.
\emd
Clearly $f\in \ZF(\DR^2_\Delta)$ and \e{d40} holds for any $x\in\ZR^2$.
\end{proof}

%% file: sections/3.4.tex
\begin{proof}[Proof of \trm{quasi-equiv}]
First, let us suppose that
\md0
\ZF^+(\ZM_1)\setminus\ZF^+(\ZM_2)\neq\varnothing.
\emd

That means there exists a non-negative function $f\in L_{\rm loc}(\ZR^n)$ such that
\md3
&\delta_{\ZM_1}(x, f) = 0,\quad \text{a.e.},\label{xd5-}\\
&\delta_{\ZM_2}(x, f) > 0,\quad x\in E_1,\label{d4-}
\emd
where $|E_1|>0$. From \e{d4-} it follows that there exist such positive numbers $\alpha$ and $\gamma$ that the set
\md1
E_2 = \{x\in\ZR^n\colon \delta_{\ZM_2}(x,f) > \alpha,\, 0\le f(x)\le \gamma\}
\emd
has positive measure. Set $f=f_{\gamma}+f^{\gamma}$, where
\md0
f_{\gamma}(x) = 
\begin{cases}
f(x),& \text{ if } 0\le f(x)\le\gamma,\\
0,& \text{ if } f(x)>\gamma.
\end{cases}
\emd
Since $\ZM_2\subseteq\ZM$ is density basis, then it differentiates $L^{\infty}$ and therefore differentiates $f_{\gamma}\in L^{\infty}$, namely we have $\delta_{\ZM_2}(x,f_{\gamma})=0$ almost everywhere. Denote by $E_3$ the subset of $E_2$ where $\delta_{\ZM_2}(x,f_{\gamma})=0$. Clearly $|E_3|=|E_2|>0$. From this we can deduce that if $x\in E_3\subset E_2$ then $\delta_{\ZM_2}(x,f) = \delta_{\ZM_2}(x,f^{\gamma})$ and $f^{\gamma}(x)=0$, since $0\le f(x)\le\gamma$. Furthermore, using \e{xd5-}, we get set $E_4\subset E_3$ of positive measure such that for any $x\in E_4$
\md3
&\delta_{\ZM_2}(x,f^{\gamma}) > \alpha,\quad f^{\gamma}(x)=0,\label{d42--}\\
&\delta_{\ZM_1}(x,f^{\gamma}) = 0,\label{d42-}
\emd
According to \e{d42-} for any  $x\in E_4$ one can choose a number $\delta(x)>0$ such that
the conditions
\md0
x\in R\in \ZM_1, \quad \diam(R)<\delta(x),
\emd
imply
\md0
\intmean{f^{\gamma}}{R} < \eta,
\emd
where $\eta>0$ will be conveniently chosen later. For some $\delta>0$ the set $G=\{x\in E_4\colon \delta(x)\ge\delta\}$ has positive measure. Thus, we have transformed \e{d42--} and \e{d42-} into
\md5
\delta_{\ZM_2}(x,f^{\gamma}) > \alpha,\, f^{\gamma}(x)=0,\quad \text{if } x\in G,\label{E-1-}\\
\intmean{f^{\gamma}}{R} < \eta,\text{ if } R\cap G\neq\varnothing,\, R\in \ZM_1,\, \diam(R)<\delta.\label{E-2-}
\emd
Since $\ZM$ differentiates $\ZI_G$, hence we may fix $x_0\in G$ with
\begin{equation*}
\lim_{\diam(R)\to 0,\, x_0\in R\in\ZM} \frac{|R\cap G|}{|R|} = 1,
\end{equation*}
which means that for any $\varepsilon>0$ there exists $\sigma(\varepsilon)$ such that $\diam(R)<\sigma(\varepsilon)$ and $x_0\in R\in \ZM$  imply $|R\cap G|>(1-\varepsilon)|R|$.
Using this relation and  \e{E-1-}, we can fix such $R$ that
\md5
x_0\in R\in \ZM_2,\quad \diam(R)<\frac{1}{c}\min\left(\sigma\left(\frac{1}{c}\right), \delta\right),\label{d10-}\\
\intmean{f^{\gamma}}{R} > \alpha,\label{xd11-}
\emd
As we have that basis $\ZM_2$ is quasi-coverable with $\ZM_1$, then for $R\in\ZM_2$ we can fix $R'\in\ZM$ and $R_k\in\ZM_1,\,k=1, 2, \dots, p$ such that \e{qc1},\e{qc2} and \e{qc0} hold. From this and \e{d10-} we get
\md0
x_0\in R'\in\ZM, \quad \diam(R')<\sigma\left(\frac{1}{c}\right),
\emd
which implies
\md1
|R'\cap G|> \left(1-\frac{1}{c}\right)|R'|\label{d13}.
\emd
which together with \e{qc2} gives that there exists $x_k\in R_k\cap G,\, k=1, 2, \dots, p$.
Now, since each $R_k$ contains some point from $G$, we can use \e{E-2-} and come to contradiction against \e{xd11-}. Namely, combining \e{E-2-}, \e{d10-} and \e{qc2} we have $x_k\in R_k\cap G,\, R_k\in\ZM_1,\,\diam(R_k)<\delta$ and therefore
\md0
\intmean{f^{\gamma}}{R_k}<\eta,\quad k=1, 2, \dots, m
\emd
which together with \e{qc0} implies
\md2
\int_{\tilde{R}}f^{\gamma}(u)\,du &\le \int_{\tilde{R}}\sum_{k=1}^p \ZI_{R_k}(u) f^{\gamma}(u)\,du\\
&= \sum_{k=1}^p \int_{R_k} f^{\gamma}(u)\,du < \eta\sum_k |R_k| \le \eta c |\tilde{R}|
\emd
and
\md0
\intmean{f^{\gamma}}{\tilde{R}} < \eta c.
\emd
On the other hand, from non-negativity of function $f^{\gamma}$ and from \e{xd11-},\e{qc0} it follows
\md0
\intmean{f^{\gamma}}{\tilde{R}}\ge \frac{|R|}{|\tilde{R}|}\cdot\intmean{f^{\gamma}}{R} > \frac{\alpha}{c},
\emd
which is impossible if choose $\eta<\frac{\alpha}{c^2}$. Thus we have proved that $\ZF^+(\ZM_1)\subset \ZF^+(\ZM_2)$.

In the same way we can prove the inverse inclusion $\ZF^+(\ZM_2)\subset \ZF^+(\ZM_1)$. Therefore the theorem is proved.
\end{proof}

%% file: sections/3.5.tex
In this section we give several corollaries from \trm{quasi-equiv} for bases formed of rectangles. First of all, notice that if we change the sets of some basis $\ZM$ by arbitrary sets of measure zero, then we get a new basis $\tilde{\ZM}$ with the same differentiation properties as $\ZM$. In particular, $\delta_{\ZM}(x,f)=\delta_{\tilde\ZM}(x,f)$ for any $x\in\ZR^n$ and $f\in L_{\rm loc}(\ZR^n)$. The reason for this is that we use Lebesgue integral, which is consistent if we modify the domain of integration by a set of measure zero. Hence, we can extend \trm{quasi-equiv} for bases formed of bounded sets, which are open up to a set of measure zero.

It is well known that the basis of all rectangles $\MR^n$ differentiates $L^\infty(\ZR^n)$, i.e. it is a density basis. Therefore we can apply \trm{quasi-equiv} for $\ZM=\MR^n$ and get a criteria for two bases formed of rectangles differentiating the same class of non-negative functions:
\begin{corollary}
If bases $\MR_1$ and $\MR_2$ formed of rectangles in $\ZR^n$ are quasi-equivalent, then $\ZF^+(\MR_1) = \ZF^+(\MR_2)$.
\end{corollary}

Let $\Omega=\{\omega_k^i\}_{i,k=1}^{n,\infty}$ be finite family of sequences with
\md1\label{Omega}
\omega_k^i\to 0 \text{ as } k\to \infty \text{ for } i=1, 2, \ldots,n.
\emd
Define the basis $\MR_\Omega$ as a family of rectangles of the form
\md0
\prod_{i=1}^n \left[(m_i-1) \omega_{k_i}^i,m_i \omega_{k_i}^i\right),\quad m_i\in\ZZ,\, k_i\in\ZN,\,i=1, 2, \ldots, n.
\emd

and the basis $\tilde{\MR}_\Omega$ as a family of rectangles with side lengths $l_i, i=1, 2, \ldots, n$ satisfying
\md0
c_1\cdot\omega_{k_i}^i \le l_i \le c_2\cdot\omega_{k_i}^i,\quad k_i\in\ZN,\, i=1, 2, \ldots, n.
\emd

Then, it can be shown that the bases $\MR_\Omega$ and $\tilde{\MR}_\Omega$ are quasi-equivalent. Therefore
\begin{corollary}
For any $\Omega$ with \e{Omega}
\md0
\ZF^+(\tilde{\MR}_\Omega)= \ZF^+(\MR_\Omega).
\emd
\end{corollary}

\begin{corollary}
If the family of sequences $\Omega$ satisfies
\md1\label{gamma}
\max_{1\le i\le n}\sup_{k\in\ZN} \frac{\omega_k^i}{\omega_{k+1}^i} < \infty,
\emd
then 
\md1\label{corOmega}
\ZF^+(\MR_\Omega) = \ZF^+(\zR^n).
\emd
\end{corollary}
\begin{proof}
Denote by $\gamma$ the finite quantity of the left hand side of \e{gamma}. Then for coefficients $c_1 = 1$ and $c_2 = \gamma+1$ we have $\ZF^+(\tilde{\MR}_\Omega) = \ZF^+(\zR^n)$. Hence from the theorem we deduce \e{corOmega}.
\end{proof}

Finally, if we take $\omega_k^i=2^{-\nu_k},\,k\in\ZN,\,i=1, 2, \ldots, n$, where $\Delta=\{\nu_k:\,k\ge1\}$ is an increasing sequence of positive integers, the basis $\MR_\Omega$ becomes the basis of all dyadic rectangles $\DR^n_\Delta$ corresponding to the sequence $\Delta$.
\begin{corollary}
If the sequence $\Delta=\{\nu_k\}$ satisfies $\gamma_\Delta<\infty$, then
\md0
\ZF^+(\DR^n_\Delta) = \ZF^+(\zR^n).
\emd
Particularly, if we take $\Delta = \ZN$, we get \otrm{zerekidze}.
\end{corollary}

%% file: sections/conclusion.tex
The thesis comprises three chapters.

In Chapter 1, it is investigated generalizations of the theorem of Fatou for convolution type integral operators with general approximate identities. It is introduced $\lambda(r)-$convergence, which is a generalization of non-tangential convergence in the unit disc. The connections between general approximate identities and optimal convergence regions for such operators are described in different functional spaces.

\begin{itemize}
\item[1.] It is found a necessary and sufficient condition on $\lambda(r)$ that ensures almost everywhere $\lambda(r)-$convergence for convolution type integral operators in both spaces of bounded measures and integrable functions. Moreover, in the case of bounded measures, the convergence occurs at any point where the measure is differentiable. In the case of integrable functions, the convergence occurs at any Lebesgue point of the function.
\item[2.] It is discovered a necessary and sufficient condition on $\lambda(r)$ that provides almost everywhere $\lambda(r)-$convergence for the same convolution type integral operators in the space of essentially bounded functions. Additionally, the convergence occurs at any Lebesgue point of the function.
\end{itemize}

In Chapter 2, it is studied some generalizations of the theorem of Littlewood, which makes an important complement to the theorem of Fatou, constructing analytic function possesing almost everywhere divergent property along a given tangential curve. The same convolution type integral operators are considered with more general kernels than approximate identities. Two kinds of generalizations of the theorem of Littlewood are obtained possessing everywhere divergent property.

\begin{itemize}
\item[3.] Under general assumptions, it is constructed a characteristic function such that the convolution with general kernels possesses everywhere divergent property along a given tangential curve. Particularly, it is proved that there exists a bounded harmonic function having everywhere strong divergent property along a given tangential curve.
\item[4.] Under general assumptions, it is constructed a bounded function, which is the boundary values of some Blaschke product, such that the convolution with general kernels owns everywhere divergent property along a given tangential curve.
\end{itemize}

Chapter 3 is devoted to some questions of equivalency of differentiation bases in $\ZR^n$. The full equivalence of basis of rare dyadic rectangles and the basis of complete dyadic rectangles in $\ZR^2$ is investigated. It is introduced quasi-equivalence between two differentiation bases in $\ZR^n$ and is considered the set of functions that such bases differentiate.

\begin{itemize}
\item[5.] It is found a necessary and sufficient condition for the full equivalence of basis of rare dyadic rectangles and the basis of complete dyadic rectangles in $\ZR^2$.
\item[6.] It is proved that two quasi-equivalent bases of some density basis in $\ZR^n$ differentiate the same set of non-negative functions.
\end{itemize}

%% file: sections/references.tex
% \vspace{40pt}